\documentclass[a4paper,11pt]{article}
\usepackage[export]{adjustbox}
\usepackage{aligned-overset}
\usepackage{amsmath,amsthm}
\usepackage{authblk}
\usepackage{booktabs}
\usepackage{cancel}
\usepackage{comment}
\usepackage{empheq}
\usepackage[hmargin={26mm,26mm},vmargin={30mm,35mm}]{geometry}
\usepackage{hyperref}
\usepackage{rotating}
\usepackage{siunitx}
\usepackage{xcolor}

\usepackage{newtxtext}
\usepackage{newtxmath}
\usepackage{subcaption}
\usepackage{multirow}

\usepackage{pgfplots}
\pgfplotsset{compat=1.18}

\allowdisplaybreaks

\newcommand{\email}[1]{\href{mailto:#1}{#1}}

\newcommand{\tnorm}[1]{{\vert\kern-0.25ex\vert\kern-0.25ex\vert #1
    \vert\kern-0.25ex\vert\kern-0.25ex\vert}_h}

\newcommand{\nTF}{ {n}_{TF}}

\newcommand{\hu}[1]{\widehat{\underline{#1}}}

\newcommand{\Err}[1]{\mathcal{E}_{#1,h}}

\newcommand{\term}{\mathfrak{T}}

\newcommand{\U}{\underline{U}^k_h}
\newcommand{\Ub}{\underline{U}^k_{h,0}}
\newcommand{\Zh}{\underline{Z}^k_{h,0}}
\renewcommand{\P}{\underline{P}^{k}_h}
\newcommand{\Pb}{\underline{P}^{k}_{h,0}}
\newcommand{\uh}{\underline{u}_h}
\newcommand{\bh}{\underline{b}_h}
\newcommand{\vh}{\underline{v}_h}
\newcommand{\wh}{\underline{w}_h}
\newcommand{\ph}{\underline{p}_h}
\newcommand{\qh}{\underline{q}_h}

\newcommand{\eu}{\underline{e}_{u,h}}
\newcommand{\eb}{\underline{e}_{b,h}}

\newcommand{\Irt}{ {I}^{k+1}_{ {\mathcal{RT\!N}},T}}

\newcommand{\Iu}{\underline{ {I}}^{k}_{ {U},h}}
\newcommand{\Ip}{\underline{I}^{k}_{P,h}}

\renewcommand{\th}{t_{h}}

\newcommand{\curl}{ {\nabla} \times}
\newcommand{\divergence}{ {\nabla} \cdot}
\newcommand{\grad}{ {\nabla}}

\newcommand{\Pe}{\mathrm{Pe}}

\theoremstyle{plain}
\newtheorem{theorem}{Theorem}

\newtheorem{lemma}[theorem]{Lemma}
\newtheorem{corollary}[theorem]{Corollary}

\theoremstyle{remark}
\newtheorem{remark}[theorem]{Remark}

\theoremstyle{definition}

\newcommand{\st}{\;:\;}

\newcommand{\Real}{\mathbb{R}}

\newcommand{\Relative}{\mathbb{Z}}

\newcommand{\Th}{\mathcal{T}_h}
\newcommand{\Fh}{\mathcal{F}_h}
\newcommand{\Fhb}{\mathcal{F}^{\rm b}_h}
\newcommand{\FT}{\mathcal{F}_T}

\newcommand{\Poly}[1]{\mathcal{P}^{#1}}
\newcommand{\RTN}[1]{\mathcal{RT\!N}^{#1}}

\newcommand{\UT}{\underline{U}_T^k}
\newcommand{\Uh}{\underline{U}_h^k}
\newcommand{\UhZ}{\underline{U}_{h,0}^k}

\newcommand{\PT}{\underline{P}_T^k}

\newcommand{\PhZ}{\underline{P}_{h,0}^k}

\newcommand{\IUh}{\underline{I}_{U,h}^k}

\newcommand{\IUT}{\underline{I}_{U,T}^k}

\newcommand{\IRTNT}[1]{I_{\mathcal{RT\!N},T}^{#1}}

\newcommand{\lproj}[2]{\pi^{#1}_{#2}}

\newcommand{\GT}{G_T^k}
\newcommand{\pT}{p_T^{k+1}}

\newcommand{\tF}{t_{\rm F}}

\newcommand{\norm}[2]{\|#2\|_{#1}}
\newcommand{\seminorm}[2]{|#2|_{#1}}

\newcommand{\Seminorm}[2]{\left|#2\right|_{#1}}

\usepackage[normalem]{ulem}
\definecolor{violet}{rgb}{0.580,0.,0.827}

\title{A Reynolds- and Hartmann-semirobust hybrid method for magnetohydrodynamics}
\author{Daniele A. Di Pietro}
\author{J\'er\^{o}me Droniou}
\author{Vito Patierno}
\affil{IMAG, Univ Montpellier, CNRS, Montpellier, France\\
  \email{daniele.di-pietro@umontpellier.fr}, %
  \email{jerome.droniou@cnrs.fr}
  \email{vito.patierno@umontpellier.fr}%
}

\begin{document}
\maketitle


\begin{abstract}
  We propose and analyze a new method for the unsteady incompressible magnetohydrodynamics equations on convex domains with approximations of the velocity, pressure and magnetic field
  inspired by Hybrid High-Order (HHO) methods.
  The proposed method is convection-semirobust, meaning that, for sufficiently smooth solutions, one can derive a priori estimates for the velocity and the magnetic field that do not depend on the inverse of the diffusion coefficients.
  This is achieved while providing relevant additional features, namely an improved order of convergence in the (asymptotic) diffusion-dominated regime as well as features common to HHO-type methods, including
  a small stencil (owing to the absence of inter-element penalty terms) %
  and the possibility to significantly reduce the size of the algebraic problems through static condensation.
  The theoretical results are confirmed by a comprehensive set of numerical experiments.
\end{abstract}


\section{Introduction}
Magnetohydrodynamics (MHD) is the branch of physics that studies the dynamics of electrically conducting fluids in the presence of magnetic fields \cite{Moreau:93}. By combining the principles of fluid mechanics and electromagnetism, MHD describes the interaction between fluid motion, induced electric currents, and the resulting Lorentz forces. This coupling gives rise to a wide range of physical phenomena that are not encountered in conventional fluid flows.
Magnetohydrodynamics has numerous applications in science and engineering. In astrophysics and geophysics, it is used to model the dynamics of stellar interiors and the solar wind. In engineering, it plays a key role, e.g., in the design of liquid-metal cooling systems for advanced nuclear reactors, electromagnetic casting and crystal growth in metallurgical processes, and plasma confinement in controlled nuclear fusion devices. Other important applications include electromagnetic pumps and flow control in liquid-metal technologies.

In this work, combining ideas from~\cite{Beirao-da-Veiga.Di-Pietro.ea:25,Droniou.Yemm:23}, we propose a new method for the unsteady incompressible MHD equations on convex domains with hybrid approximations of
all the involved fields (velocity and magnetic fields as well as fluid and magnetic pressures).
The key novelty of the proposed method is that it delivers approximations of both the velocity and the magnetic field that are quasi-robust with respect to dominant convection.
Quasi-robustness refers here to the fact that, assuming regularity of the exact solution, the error estimates are uniform in both the viscosity and the magnetic diffusivity.
This is achieved while providing relevant additional features, namely an improved order of convergence in the (asymptotic) diffusion-dominated regime, %
a small stencil (owing to the absence of inter-element penalty terms), %
and the possibility to significantly reduce the size of the algebraic problems through static condensation.
\smallskip

Denote by $H^1(\Omega)$, $H(\operatorname{curl};\Omega)$, and $H(\operatorname{div};\Omega)$ the usual Hilbert spaces of functions that are square-integrable along with, respectively, their gradient, curl, or divergence.
Considering Maxwell's equations with homogeneous boundary conditions on a domain $\Omega$, the natural functional space for the magnetic field $b$ is either $H(\text{curl};\Omega)\cap H_0(\text{div};\Omega)$ or $H_0(\text{curl};\Omega)\cap H(\text{div};\Omega)$, where the subscript indicates the zero-trace subspace.
It is known from~\cite{Amrouche.Bernardi.ea:98} that, if $\Omega\subset \Real^3$ is a convex polyhedron, both the previous spaces are continuously embedded in $H^1(\Omega)^3$.
In this case, one can envisage weak formulations with $b$ in $H^1(\Omega)^3$.
In particular, as observed, e.g., in~\cite{Droniou.Yemm:23}, in this case the curl-curl operator can be replaced by the vector Laplacian, an idea which is taken up in the present work.
\smallskip

The numerical approximation of the MHD equations has been considered in several works starting from the early 1990s.

In~\cite{Gunzburger.Meir.ea:91}, the authors consider a reformulation of the steady problem with both the velocity and magnetic fields in $H^1(\Omega)^3$ and prove optimal error estimates for conforming finite element approximations.
In \cite{Schotzau:04}, on the other hand, a formulation with magnetic field in $H(\operatorname{curl}; \Omega) \cap H(\operatorname{div}; \Omega)$ is used as a starting point for a scheme based on Lagrange elements for the velocity and edge Nédélec elements for the magnetic field.
Both papers work under a small-data assumption, which is classically needed for the analysis of the steady problem and limits the applicability of the theoretical results to the diffusion-dominated regime.
This assumption can be relaxed when compactness arguments are used for convergence and/or when the unsteady version of the MHD problem is considered, as is the case, e.g., in \cite{Prohl:08,Droniou.Yemm:23}.

Yielding a globally divergence-free velocity field is a desirable property for a numerical scheme, both from the practical standpoint (when, e.g., the velocity field is used in the passive advection of a contaminant) and from the theoretical one (since $H(\operatorname{div};\Omega)$-conformity is linked to properties such as pressure-robustness~\cite{Linke:14}).
In~\cite{Greif.Li.ea:10}, a globally divergence-free discrete velocity is obtained considering a scheme based on a N\'ed\'elec approximation of the magnetic field together with a Brezzi--Douglas--Marini (BDM) approximation of the velocity with penalization of tangential jumps.
The hybridizable discontinuous Galerkin scheme proposed and analyzed in~\cite{Qiu.Shi:19}, generalizing the approach of~\cite{Houston.Schotzau.ea:09} to the nonlinear case, also yields an $H(\operatorname{div};\Omega)$-conforming velocity field.

Recent contributions put the emphasis on robustness with respect to the physical parameters.
In \cite{Beirao-da-Veiga.Dassi.ea:24}, the authors propose a pressure-robust stabilized method relying on a BDM approximation of the velocity and a Lagrange approximation of the magnetic flux.
A similar choice of spaces is considered in~\cite{Beirao-da-Veiga.Dassi.ea:25} for the fully nonlinear problem, along with two different strategies for the enforcement of the solenoidal condition on the magnetic field.
The resulting schemes are both pressure-robust and convection-quasi-robust.
For the version of the methods corresponding to a polynomial degree $k$, the energy error converges with the mesh size at a rate $k$ in the asymptotic (diffusion-dominated) regime, with an improvement up to $k+\frac12$ in the pre-asymptotic (convection-dominated) regime for the version of the method where the solenoidal constraint is enforced through a Lagrange multiplier.
Notice that the topics of convection-quasi-robustness and regime-dependent estimates, along with the corresponding analysis techniques, have only recently emerged in the literature.
Such analysis techniques could in all likelihood be applied to older schemes, but this lies outside the scope of the present work.
For this reason, we limit the discussion of these aspects to \cite{Beirao-da-Veiga.Dassi.ea:24,Beirao-da-Veiga.Dassi.ea:25}, where they are explicitly addressed.
\smallskip

The main contribution of the present work is the introduction and convergence analysis of a method which matches the robustness properties of~\cite{Beirao-da-Veiga.Dassi.ea:25}, and has two additional appealing features:
\begin{itemize}
\item the pre-asymptotic and asymptotic convergence rates are, respectively, $k+\frac12$ and $k+1$, i.e., the convergence rate improves with mesh refinement. This is made possible by the use of techniques inspired by Hybrid High-Order (HHO) methods~\cite{Di-Pietro.Ern.ea:14,Di-Pietro.Ern:15,Di-Pietro.Droniou:20} for the discretization of the diffusive terms;
\item it does not require inter-element jump penalization, resulting in a smaller stencil as well as the possibility to statically condense the internal degrees of freedom in the spirit of~\cite[Section~6.2]{Di-Pietro.Ern.ea:16*1}.
\end{itemize}
The convection- or diffusion-dominated regime is identified by four dimensionless numbers, of Reynolds- or Hartmann-type.
Starting from the reformulation of the MHD equations considered in~\cite{Droniou.Yemm:23}, and taking inspiration from the HYbrid PREssure (HYPRE) methods of~\cite{Botti.Botti.ea:26,Beirao-da-Veiga.Di-Pietro.ea:25} (which, in turn, are inspired by the classical Botti--Massa method of~\cite{Botti.Massa:22}), we approximate all the vector- and scalar-valued variables in hybrid spaces.
We specifically consider the spaces discussed in~\cite[Section~5.4]{Botti.Botti.ea:26}, which rely on the use of Raviart--Thomas--N\'ed\'elec vector-valued variables inside each element.
A careful selection of the trace spaces enforces continuity of the normal components of these variables across interfaces and thus yields $H(\operatorname{div};\Omega)$-conforming element-based approximations of both the velocity and the magnetic field.
The convergence analysis takes inspiration from the techniques developed in~\cite{Di-Pietro.Droniou.ea:15,Botti.Di-Pietro.ea:18,Di-Pietro.Droniou:23*2,Beirao-da-Veiga.Di-Pietro.ea:24,Beirao-da-Veiga.Di-Pietro.ea:25}, where regime-dependent estimates of the contributions to the consistency error are derived based on the value of suitable dimensionless numbers.
This results in error estimates that are robust across the entire range of physically relevant regimes, from convection- to diffusion-dominated.
Since they are obtained by assuming (rather than proving) additional regularity of the exact solution, estimates of this type are usually referred to in the literature as ``quasi-robust''.
The theoretical results are supported by a comprehensive set of two- and three-dimensional numerical tests which, in particular, confirm the predicted orders of convergence in the various regimes.
\smallskip

The rest of the paper is organized as follows.
In Section~\ref{sec:continuous} we recall the formulation of the continuous problem.
Section~\ref{sec:discrete.setting} describes the discrete setting (mesh, discrete spaces).
The discrete problem is stated in Section~\ref{sec:discrete.problem+main.results}, where we also state the main results of the analysis.
A comprehensive set of two- and three-dimensional numerical results is provided in Section~\ref{sec:numerical.results}.
Finally, the proofs of the main theoretical results are given in Section~\ref{sec:analysis}.


\section{Continuous problem}\label{sec:continuous}

Let $\Omega \subset \Real^3$ denote an open, bounded, convex polyhedral domain with Lipschitz boundary $\partial\Omega$ and outward unit normal vector $n$.
We consider a fluid with constant viscosity $\nu$ and magnetic diffusivity $\mu$, source terms
$f, g: \lparen 0,\tF\rbrack \times \Omega \rightarrow \Real^3$
such that $\divergence g = 0$, and
divergence-free initial conditions $u_0, b_0 \in H^1_0(\Omega)^3$.
The unsteady MHD problem reads as follows:
Find the velocity $u : [0, \tF] \times \Omega \to \Real^3$,
the magnetic field $b : [0, \tF] \times \Omega \to \Real^3$,
and the pressure  $p :\lparen 0,\tF\rbrack\times \Omega \rightarrow \Real$ such that
\begin{subequations}\label{eq:strong}
  \begin{equation}\label{eq:strong:ic}
    u(0,\cdot)=u_0 \text{ and } b(0,\cdot)=b_0,
  \end{equation}
  and, for $t \in \lparen 0, \tF\rbrack$,
  \begin{alignat}{4}\label{eq:strong:momentum}
    \partial_{t} u(t)
    - \nu \Delta u(t)
    + (u(t) \cdot  \nabla) u(t)
    -  (\curl b (t)) \times b(t)
    + \nabla p(t) &= f(t)
    &\qquad& \text{in $\Omega$},
    \\ \label{eq:strong:maxwell}
    \partial_{t} b(t)
    + \mu \curl \curl b(t)
    -\curl(u(t) \times b(t))
    &= g(t)
    &\qquad& \text{in $\Omega$},
    \\ \label{eq:strong:mass}
    \divergence u(t) &= 0
    &\qquad& \text{in $\Omega$},
    \\ \label{eq:strong:gauge}
    \divergence b(t) &= 0
    &\qquad& \text{in $\Omega$},
  \end{alignat}
\end{subequations}
completed with appropriate boundary conditions.
Above, given a function of time and space $\psi$, we have adopted the convention that $\psi(t)$ stands for the function of space only $\psi(t,\cdot)$.

Taking inspiration from~\cite{Droniou.Yemm:23} and assuming, for the time being, sufficient regularity for the following manipulations to make sense, the following identities hold in $[0,\tF] \times \Omega$:
\begin{subequations}\label{eq:remove.curl}
  \begin{gather}\label{eq:remove.curl:curl.curl.b}
    \curl \curl b
    = \cancel{\grad (\nabla \cdot b)} - \Delta b,
    \\ \label{eq:remove.curl:curl.u.times.b}
    -\curl( u \times b )
    = ( u \cdot \grad) b
    - ( b \cdot \grad) u
    - \cancel{ u (\nabla \cdot b)}
    + \cancel{ b(\nabla \cdot u), }
    \\ \label{eq:remove.curl:curl.b.times.b}
    (\curl b) \times b
    = ( b \cdot \grad) b
    - \frac12 \grad |b|^2,
  \end{gather}
\end{subequations}
where we have used~\eqref{eq:strong:mass} and~\eqref{eq:strong:gauge} in the cancellations.

To write a variational formulation of problem~\eqref{eq:strong}, we introduce a Lagrange multiplier $r$ to enforce the divergence-free constraint~\eqref{eq:strong:gauge}.
This new variable can be regarded as a magnetic pressure incorporating the contribution $\frac12 |b|^2$ in ~\eqref{eq:remove.curl:curl.b.times.b} in the same spirit as the usual kinematic pressure.
Accounting for \eqref{eq:remove.curl}, problem \eqref{eq:strong} becomes:
Find $u$, $b$, and $p$ as above, as well as $r : \lparen 0,\tF\rbrack \times \Omega \to \Real$ such that, for $t \in \lparen 0, \tF\rbrack$,
\begin{subequations}\label{eq:modified}
  \begin{alignat}{4}\label{eq:modified:momentum}
    \partial_t u(t)
    - \nu \Delta u(t)
    + (u(t) \cdot  \grad) u(t)
    - (b(t) \cdot \grad) b(t)
    + \nabla p(t) &= f(t)
    &\qquad& \text{in $\Omega$},
    \\ \label{eq:modified:maxwell}
    \partial_t b(t)
    - \mu \Delta b(t)
    + (u(t) \cdot \grad) b(t)
    - (b(t) \cdot \grad) u(t)
    + \nabla r(t) &= g(t)
    &\qquad& \text{in $\Omega$},
    \\ \label{eq:modified:mass.gauge}
    \divergence u(t)
    = \divergence b(t)
    &= 0
    &\qquad& \text{in $\Omega$},
    \\ \label{eq:modified:bc}
    u(t) = b(t) &= 0 &\qquad& \text{on $\partial \Omega$},
    \\ \label{eq:modified:mean.p.r}
    \int_\Omega p(t)
    = \int_\Omega r(t) &= 0,
  \end{alignat}
\end{subequations}
completed with the initial condition \eqref{eq:strong:ic}.

A few remarks are in order.

\begin{remark}[Space and boundary conditions for the magnetic field]
  In \eqref{eq:remove.curl:curl.curl.b}, the curl-curl operator has been replaced by a vector Laplacian.
  As a consequence, in the variational formulation of problem \eqref{eq:modified}, we take the magnetic field $b$ in $H^1(\Omega)^3$ instead of $H(\operatorname{curl};\Omega) \cap H(\operatorname{div};\Omega)$.

  In order to simplify the exposition and keep the focus on Reynolds- and Hartmann-semi-robustness, we have additionally taken homogeneous Dirichlet boundary conditions for both the velocity and the magnetic fields.
  This makes it possible, in particular, to use the same consistency results for terms of similar nature.
  A physically more accurate choice would have been to enforce $b(t) \cdot n = 0$ and $n\times(\nabla\times b(t))=0$ on $\partial \Omega$.
  With this choice, the magnetic field is still in $H^1(\Omega)^3$ owing to the continuous embedding of $H(\operatorname{curl};\Omega) \cap H_0(\operatorname{div};\Omega)$ into this space for convex domains; see~\cite{Amrouche.Bernardi.ea:98}.
  The minor (yet tedious) modifications required to accommodate these boundary conditions on $b$ can be inferred from~\cite{Droniou.Yemm:23}; see also~\cite{Potherat.Sommeria.ea:02} for a more general discussion of boundary conditions for MHD.
\end{remark}

\begin{remark}[Validity of \eqref{eq:modified} in two space dimensions]\label{rem:2d}
  From the mathematical perspective, the modified formulation~\eqref{eq:modified} also makes sense in two space dimensions, a fact that we will exploit in the numerical tests of Section~\ref{sec:numerical.results:2d}.
\end{remark}

In view of Remark~\ref{rem:2d}, from this point we consider problem \eqref{eq:weak} in dimension $d \in \{2, 3\}$ (for $d = 2$, $\Omega$ is a bounded convex polygon).
Define the bilinear forms
$a : H^1(\Omega)^d \times H^1(\Omega)^d \to \Real$
and $B : H^1(\Omega)^d \times L^2(\Omega) \to \Real$
as well as the trilinear form $t : [ H^1(\Omega)^d ]^3 \to \Real$
such that, for all $(v, w, z) \in [ H^1(\Omega)^d ]^3$ and all $q \in L^2(\Omega)$,
\[
a(w,v) \coloneqq \int_\Omega \grad w : \grad v,
\quad
t(v,w,z) \coloneqq \int_\Omega ( v \cdot \grad) w \cdot z,
\quad
B(v,q) \coloneqq - \int_\Omega (\divergence v)\,q.
\]
Denote by $H^1_0(\Omega)$ the subspace of $H^1(\Omega)$ consisting of functions with vanishing trace on $\partial \Omega$, and let $L^2_0(\Omega) \coloneqq \left\{ q \in L^2(\Omega) \;:\; \int_\Omega q = 0\right\}$.
Assume
$f \in C^{0}(\lparen 0,\tF\rbrack; L^2(\Omega)^d)$ and $g \in C^0(\lparen 0,\tF\rbrack;H(\operatorname{div};\Omega))$
such that $\divergence g = 0$.
The variational formulation reads:
Find $(u,b,p,r) \in [C^{1}([0,\tF]; H^{1}_0(\Omega)^d)]^2 \times [C^{0}(\lparen 0,\tF\rbrack; L^2_{0}(\Omega))]^2$ such that, for all $t \in \lparen 0, \tF\rbrack$ and all $(v, w, q) \in [ H_0^1(\Omega)^d ]^2 \times L^2_0(\Omega)$,%
\begin{equation}\label{eq:weak}
  \begin{alignedat}{4}
    \int_\Omega \partial^{}_{t} u(t) \cdot v
    + \nu a(u(t),v)
    + t(u(t),u(t),v)
    - t(b(t),b(t),v)
    + B(v,p(t))
    &= \int_\Omega f(t) \cdot v,
    \\
    \int_\Omega \partial_t b(t) \cdot w
    + \mu a(b(t),w)
    + t(u(t),b(t),w)
    - t(b(t),u(t),w)
    + B(w,r(t))
    &= \int_\Omega g(t) \cdot w,
    \\
    B(u(t),q)
    = B(b(t),q) &= 0,
  \end{alignedat}
\end{equation}
completed with the initial condition \eqref{eq:strong:ic}.


\section{Discrete setting}\label{sec:discrete.setting}

\subsection{Mesh}

Denote by $\Th$ a conforming simplicial mesh of $\Omega$, assumed to belong to a regular family in the sense of \cite{Ciarlet:02}.
The associated set of simplicial faces is $\Fh$ and we denote by $\Fhb$ its subset of boundary faces contained in $\partial \Omega$.
Notice that the term ``faces'' refers here to $(d-1)$-dimensional simplices, i.e., faces if $d = 3$ and edges if $d = 2$.
For each element $T \in \Th$, we define $\FT \subset \Fh$ as the set of faces contained in $\partial T$ and, for any $F \in \FT$, we denote by $n_{TF}$ the unit normal vector to $F$ pointing out of $T$.
The diameter of a mesh entity $Y \in \Th \cup \Fh$ is denoted by $h_Y$, so that the mesh size is $h \coloneqq \max_{T \in \Th} h_T$.

In what follows, $x \lesssim y$ means $x \le C y$ with a real number $C > 0$ independent of the mesh size $h$, the viscosity $\nu$, the magnetic diffusivity $\mu$, and, for local inequalities on a mesh entity $Y \in \Th \cup \Fh$, also of $Y$.
Additional assumptions on the dependencies of $C$ will be specified whenever relevant.
Notice that $C$ may depend on quantities such as the space dimension $d$, the domain $\Omega$, the final time $\tF$, the polynomial degree, or the mesh regularity parameter.
We also write $x \simeq y$ for ``$x \lesssim y$ and $y \lesssim x$''.

\subsection{Polynomial spaces}

Let $Y \in \Th \cup \Fh$ be either a mesh element or a face, and let $\ell \ge 0$ be an integer.
We denote by $\Poly{\ell}(Y)$ the set of functions obtained by restricting to $Y$ the polynomials in the spatial variables of total degree at most $\ell$, and we conventionally set $\Poly{-1}(Y) \coloneqq \{0\}$.
The $L^2$-orthogonal projector onto $\Poly{\ell}(Y)$ is $\lproj{\ell}{Y} : L^1(Y) \to \Poly{\ell}(Y)$ such that, for all $q \in L^1(Y)$,
$\int_Y \lproj{\ell}{Y} q \, r = \int_Y q \, r$ for all $r \in \Poly{\ell}(Y)$.
For vector-valued functions, $\lproj{\ell}{Y}$ is understood to act component-wise.

Let now $T \in \Th$.
The Raviart--Thomas--Nédélec finite element space of order $\ell \ge 1$ on $T$ is
\[
\RTN{\ell}(T) \coloneqq \Poly{\ell-1}(T)^d + x \Poly{\ell-1}(T),
\]
with $x$ denoting the spatial coordinate.
Its interpolator $\IRTNT{\ell} : H^1(T)^d \to \RTN{\ell}(T)$ is uniquely defined by the following conditions:
For all $v \in H^1(T)^d$,
\begin{equation}\label{eq:IRTNT}
  \lproj{\ell-2}{T} \IRTNT{\ell} v = \lproj{\ell-2}{T} v,
  \qquad
  \IRTNT{\ell} v \cdot n_{TF} = \lproj{\ell-1}{F}(v \cdot n_{TF})
  \quad \forall F \in \FT.
\end{equation}
The following approximation results are proved in~\cite[Lemma~2.1]{Beirao-da-Veiga.Di-Pietro.ea:25} (see also, e.g., \cite[Lemma~3.17]{Gatica:14}, \cite[Proposition~2.5.1]{Boffi.Brezzi.ea:13} or \cite[Theorem~16.4]{Ern.Guermond:21*1} for the case $p = 2$):
For all $p \in [1,\infty]$ and all integers $q$ and $m$ such that $0 \le q \le \ell-1$ and $0 \le m \le q+1$, we have
\begin{equation}\label{eq:RTN_bound}
  \seminorm{W^{m,p}(T)^d}{v - \IRTNT{\ell} v}
  \lesssim h_T^{q+1-m} \seminorm{W^{q+1,p}(T)^d}{v}
  \qquad \forall v \in W^{q+1,p}(T)^d
\end{equation}
and, for each $F \in \FT$,
\begin{equation}\label{eq:RTN_bound_F}
  \norm{L^p(F)^d}{v - \IRTNT{\ell} v}
  \lesssim h_T^{q+1-\frac1p} \seminorm{W^{q+1,p}(T)^d}{v}
  \qquad \forall v \in W^{q+1,p}(T)^d.
\end{equation}

\subsection{Discrete spaces}
Given an integer $k\geq 0$, define the discrete spaces of vector- and scalar-valued fields, respectively, as
\begin{multline*}
  \U\coloneq\Big\{ \vh = ((v_T )_{T\in \Th},(v_F)_{F\in \Fh}) \;:\;
  \\
  \text{$v_T \in \RTN{k+1}(T)$  for all $T \in \Th$ and
    $v_F \in  \mathcal{P}^{k}(F)^d$ for all $F \in \Fh$}
  \Big\}
\end{multline*}
and
\begin{multline*}
  \P\coloneq\Big\{
  \underline{q}_h = ((q_T )_{T\in \Th},(q_F)_{F\in \Fh}) \;:\;
  \\
  \text{%
    $q_T \in \Poly{k}(T)$ for all $T \in \Th$
    and $q_F \in \Poly{k}(F)$ for all $F \in \Fh$
  }
  \Big\}.
\end{multline*}
The meaning of the polynomial components in these spaces is provided by the interpolators
$\Iu: H^1(\Omega)^d\to \U$ and $\Ip: H^1(\Omega)\to \P$ such that, for all $v\in H^1(\Omega)^d$ and all $q\in H^1(\Omega)$,
\begin{equation}\label{eq:Iu}
  \Iu v\coloneqq ((\Irt v)_{T\in \Th}, (\lproj{k}{F} v)_{F\in \Fh}),
  \qquad
  \Ip q\coloneqq ((\lproj{k}{T} q)_{T\in \Th}, (\lproj{k}{F} q)_{F\in \Fh}).
\end{equation}
The restrictions of the above spaces, their elements, and the interpolators to a mesh element $T$ are denoted by replacing the subscript $h$ with $T$ and are obtained by collecting the components associated with $T$ and its faces.

Given $\vh \in \U$ and $\underline{q}_h \in \P$, we define the broken polynomial functions $v_h \in L^2(\Omega)^d$ and $q_h \in L^2(\Omega)$ such that
\begin{equation}\label{eq:vh.qh}
  (v_h)_{|T} \coloneqq v_T, \qquad (q_h)_{|T} \coloneqq q_T \qquad \forall T \in \Th.
\end{equation}
The subspaces of $\U$ and $\P$ incorporating the homogeneous boundary condition for vector-valued fields and the zero-average condition for scalar-valued fields are, respectively, given by
\[
\Ub \coloneqq \left\{ \vh \in \U \;:\; \text{$v_F = 0$ for all $F \in \Fhb$} \right\},
\qquad
\Pb \coloneqq \left\{ \underline{q}_h \in \P \;:\; \int_{\Omega} q_h = 0 \right\}.
\]


\section{Discrete problem and main results}\label{sec:discrete.problem+main.results}

To obtain a numerical scheme, we replace the spaces and forms in the weak formulation~\eqref{eq:weak} with their discrete counterparts.
The relevant spaces have already been defined in the previous section.
This section contains the definitions of the forms that arise in the discretization of the various terms along with their main properties, as well as the statement of the discrete problem.

\subsection{Discrete $L^2$-like inner product and norm}
To discretize the unsteady term, we define the discrete $ L^2$-like inner product $ (\cdot,\cdot)_{0,h}:\U\times\U \to \Real$ such that, for all $(\wh,\vh)\in \U\times \U$,
\begin{equation}\label{eq:prod.0.h}
  \begin{gathered}/JCOMP/submissions
    (\wh, \vh)_{0,h} \coloneqq \sum_{T\in \Th} (\underline{w}_T,\underline{ {v}}_T)_{0,T},
    \\
    (\underline{w}_T,\underline{v}_T)_{0,T}\coloneqq \int_T w_T \cdot  {v}_T  + h_T\sum_{F\in \mathcal{F}_T} \int_F (w_F - w_T) \cdot ( {v}_F -  {v}_T).
  \end{gathered}
\end{equation}
The corresponding global and local norms are defined as follows: for all $\vh\in \U$,
\begin{equation}\label{eq:norm.0.h}
  \text{%
    $\norm{0,h}{\vh} \coloneqq (\vh, \vh)_{0,h}^{\frac12}$
    and $\norm{0,T}{\underline{v}_T} \coloneqq (\underline{v}_T, \underline{v}_T)_{0,T}^{\frac12}$
    for all $T\in \Th$.
  }
\end{equation}

\subsection{Diffusion}

Given a mesh element $T \in \Th$, we define the velocity reconstruction $\pT : \UT \to \Poly{k+1}(T)^d$ such that, for all $\underline{v}_T \in \UT$,
\begin{equation*}
  \begin{gathered}
    \int_T \nabla \pT \underline{v}_T : \nabla w
    = - \int_T v_T \cdot \Delta w
    + \sum_{F \in \FT} \int_F v_F \cdot (\nabla w \, n_{TF})
    \qquad \forall w \in \Poly{k+1}(T)^d,
    \\
    \int_T \pT \underline{v}_T =
    \begin{cases}
      \sum_{F \in \FT} \frac{d_{TF}}{d} \int_F v_F & \text{if $k = 0$},
      \\
      \int_T v_T & \text{otherwise},
    \end{cases}
  \end{gathered}
\end{equation*}
where, for any $F \in \FT$, $d_{TF}$ denotes the distance of the center of mass of $T$ from the plane containing $F$.
The diffusion bilinear form $a_h : \Uh \times \Uh \to \Real$ is such that, for all $(\wh, \vh) \in \Uh \times \Uh$,
\[
\begin{gathered}
  a_h(\wh, \vh)
  \coloneqq \sum_{T \in \Th} a_T(\underline{w}_T, \underline{v}_T),
  \\
  a_T(\underline{w}_T, \underline{v}_T)
  \coloneqq \int_T \nabla \pT \underline{w}_T : \nabla \pT \underline{v}_T
  + s_T(\underline{w}_T, \underline{v}_T).
\end{gathered}
\]
Above, $s_T : \UT \times \UT \to \Real$ denotes a stabilization bilinear form that penalizes the components of $(\delta_T^k \underline{v}_T, (\delta_{TF}^k \underline{v}_T)_{F \in \FT} ) \coloneqq \IUT \pT \underline{v}_T - \underline{v}_T$. A possible expression for $s_T$ is the following: for all $(\underline{w}_T, \underline{v}_T) \in \UT$,
\[
s_T(\underline{w}_T, \underline{v}_T)
= \lambda_T h_T^{-2} \int_T \delta_T^k \underline{w}_T \cdot \delta_T^k \underline{v}_T
+ h_T^{-1} \sum_{F \in \FT} \int_F \delta_{TF}^k \underline{w}_T \cdot \delta_{TF}^k \underline{v}_T,
\]
with, e.g., $\lambda_T \coloneqq \operatorname{card}(\FT) \frac{h_T^d}{|T|}$ to equilibrate the two contributions.
The purpose of the stabilization is to ensure the following coercivity and boundedness properties:
\begin{equation}\label{eq:ah:norm.equivalence}
  a_h(\vh, \vh)
  \simeq \norm{1,h}{\vh}^2
  \qquad \forall \vh \in \Uh,
\end{equation}
where $\norm{1,h}{\cdot}$ is the $H^1$-like seminorm such that, for all $\vh \in \U$,
\begin{equation}\label{eq:norm.1.h}
  \begin{gathered}
    \norm{1,h}{\vh}^2
    \coloneqq
    \sum_{T\in \Th} \norm{1,T}{\underline{v}_T}^2,
    \\
    \norm{1,T}{\underline{v}_T}^2
    \coloneqq
    \norm{L^2(T)^{d\times d}}{\grad v_T}^2
    + h^{-1}_T \sum_{F\in \mathcal{F}_T} \norm{L^2(F)^d}{v_F - v_T}^2.
  \end{gathered}
\end{equation}
Restricted to $\Ub$, the map $\norm{1,h}{\cdot}$ defines a norm.

\subsection{Solenoidal constraint}

Define the discrete pressure gradient $\GT: \PT \to \RTN{k+1}(T)$ such that, for all $\underline{q}_T\in \underline{P}^k_T$,
\[
\int_T \GT \underline{q}_T \cdot w
= - \int_T q_T\, (\divergence w)
+ \sum_{F\in \mathcal{F}_T} \int_F q_F\, (w\cdot n_{TF})
\qquad \forall w \in \RTN{k+1}(T).
\]
The conservation of mass and Gauss's law~\eqref{eq:modified:mass.gauge} are enforced by the bilinear form $B_h: \U \times \P \to \Real$, such that, for all $(\vh,\underline{q}_h)\in \U \times \P$,
\begin{equation*}
  B_h(\vh,\underline{q}_h)
  \coloneqq
  \sum_{T\in \Th} \int_T  {v}_T \cdot  \GT \underline{q}_T.
\end{equation*}
For future use, we define the following subspace of $\Ub$.
\begin{equation}\label{eq:ZhZ}
  \Zh \coloneqq
  \Big\{ \vh \in \UhZ \st
  \text{$B_h(\vh, \underline{q}_h)=0$ for all $\underline{q}_h \in \Pb$}
  \Big\}.
\end{equation}
A few remarks are in order.

\begin{remark}[Pointwise divergence-free vector fields]\label{rem:div-free-subspace}
  Let $\vh\in \Ub$ be such that
  \begin{equation}\label{eq:discr-solen}
    B_h(\vh,\underline{q}_h)=0 \qquad \forall \qh \in \Pb.
  \end{equation}
  Recalling \cite[Proposition~3.3]{Beirao-da-Veiga.Di-Pietro.ea:25}, \eqref{eq:discr-solen} is equivalent to the following set of conditions:
  \[
  \begin{alignedat}{4}
    \divergence v_T &= 0
    &\qquad& \forall T\in \mathcal{T}_h,
    \\
    v_{T_1} \cdot n_{T_1 F} + v_{T_2} \cdot  {n}_{T_2 F} &= 0
    &\qquad& \forall F \in \Fh \setminus \Fhb,\\
    v_T\cdot n_{TF}&=0 &\qquad& \forall T\in\Th\,,\;\forall F\in\FT\cap\Fhb
  \end{alignedat}
  \]
  where, for any $F \in \Fh \setminus \Fhb$, we have denoted by $T_1$ and $T_2$ the distinct mesh elements such that $F \subset \partial T_1 \cap \partial T_2$. These conditions imply, in particular, that the field $v_h$ defined from $\vh$ according to \eqref{eq:vh.qh} satisfies:
  \begin{equation}\label{eq:vh.zero.div}
    v_h\in H(\operatorname{div};\Omega)\,,\quad \divergence v_h=0 \text{ in $\Omega$}\,,\quad v_h\cdot n=0 \text{ on $\partial\Omega$.}
  \end{equation}
  Moreover, since, for all $T \in \Th$, $v_T$ is both in $\RTN{k+1}(T)$ and divergence-free, it holds in fact that $v_T\in \Poly{k}(T)^d$.
\end{remark}

\begin{remark}[Interpolants of divergence-free functions]\label{rem:interpolate.div.free.functions}
  Let $w \in H^1_0(\Omega)^d$ be such that $\divergence w = 0$.
  By \cite[Proposition~4.3]{Beirao-da-Veiga.Di-Pietro.ea:25}, we have $\Iu w \in \Zh$.
\end{remark}

The choice of the bilinear form $B_h$ is justified by the following result, which is a straightforward consequence of \cite[Lemma~3 and Theorem~14]{Botti.Botti.ea:26}.

\begin{lemma}[Inf-sup condition on $B_h$]
  It holds
  \begin{equation}\label{eq:inf-sup}
    \left(
    \norm{L^2(\Omega)}{q_h}^2
    + \sum_{T \in \Th} h_T^2 \norm{L^2(T)^d}{\GT \underline{q}_T}^2
    \right)^{\frac12}
    \lesssim
    \sup_{\vh \in \UhZ \setminus \{ \underline{0} \}}
    \frac{B_h( \vh, \underline{q}_h)}{\norm{1,h}{\vh}}
    \qquad \forall \underline{q}_h \in \PhZ.
  \end{equation}
\end{lemma}

\subsection{Convection}

The convective trilinear form $\th:[\U]^3 \to \Real$ is such that, for all $(\wh,\vh,\underline{z}_h)\in[\U]^3$,
\begin{equation}\label{eq:th}
  \begin{gathered}
    \th(\wh,\underline{ {v}}_h,\underline{ {z}}_h) \coloneqq \sum_{T\in \Th} t_T(\underline{w}_T,\underline{ {v}}_T,\underline{ {z}}_T),
    \\
    t_T(\underline{w}_T,\underline{ {v}}_T,\underline{ {z}}_T)\coloneqq \int_T (w_T \cdot \grad)  {v}_T \cdot  {z}_T \hspace{2pt}  + \frac12 \sum_{F \in\mathcal{F}_T} \int_F (w_T \cdot \nTF)( {v}_F -  {v}_T ) \cdot ( {z}_F+  {z}_T).
  \end{gathered}
\end{equation}
To state the properties of the convective trilinear form, we need the discrete counterpart of the $W^{1,\infty}$-seminorm such that, for all $\vh \in \U$,
\[
\begin{gathered}
  \norm{1,\infty,h}{\vh}
  \coloneqq
  \max_{T\in \Th} \norm{1,\infty,T}{\underline{v}_T},
  \\
  \norm{1,\infty,T}{\underline{v}_T}
  \coloneqq
  \norm{L^{\infty}(T)^{d\times d}}{\grad v_T}
  + h^{-1}_T \max_{F\in \mathcal{F}_T} \norm{L^{\infty}(F)^d}{v_F - v_T}.
\end{gathered}
\]
Restricted to $\Ub$, $\norm{1,\infty,h}{\cdot}$ is a norm.
Additionally, recalling~\cite[Lemma~4.2]{Beirao-da-Veiga.Di-Pietro.ea:25}, we have
\begin{equation}\label{eq:W.1.infty.boundedness.IUh}
  \norm{1,\infty,h}{\IUh w}
  \lesssim \seminorm{W^{1,\infty}(\Omega)^d}{w}
  \qquad \forall w \in W^{1,\infty}(\Omega)^d.
\end{equation}
By~\cite[Lemmas~3.5 and~4.1]{Beirao-da-Veiga.Di-Pietro.ea:25}, we have the following properties:
\begin{itemize}
\item \emph{Non-dissipativity.} For all $\wh \in \Zh$, we have
  \begin{equation}\label{eq:non-dissipativity}
    \th(\wh, \vh, \vh) = 0 \qquad \forall \vh \in \Ub;
  \end{equation}
\item \emph{Boundedness.} For all $(\wh,\vh,\underline{z}_h)\in \big[ \U\big]^3$, we have
  \begin{equation}\label{thbound}
    |\th(\wh,\vh, \underline{ {z}}_h)| \lesssim
    \norm{L^2(\Omega)^d}{w_h} \norm{1,\infty,h}{\vh} \norm{0,h}{\underline{z}_h}.
  \end{equation}
\end{itemize}
Noticing that $\norm{L^2(\Omega)^d}{w_h}\le \norm{0,h}{\underline{w}_h}$ by \eqref{eq:prod.0.h}--\eqref{eq:norm.0.h}, the bound \eqref{thbound} yields
\begin{equation}\label{thbound.0h}
  |\th(\wh,\vh, \underline{ {z}}_h)| \lesssim
  \norm{0,h}{\underline{w}_h} \norm{1,\infty,h}{\vh} \norm{0,h}{\underline{z}_h}.
\end{equation}
The following lemma generalizes \cite[Lemma~3.5]{Beirao-da-Veiga.Di-Pietro.ea:25} and will play a key role in our analysis.

\begin{lemma}[Skew-symmetry of $\th$]
  For all $(\wh,\underline{ {v}}_h,\underline{ {z}}_h) \in  \Zh\times \Ub \times \Ub $, we have
  \begin{equation}\label{eq:skew-simmetry}
    \th(\wh,\underline{ {v}}_h, \underline{ {z}}_h) =
    -\th(\wh, \underline{ {z}}_h, \underline{ {v}}_h).
  \end{equation}
\end{lemma}

\begin{proof}
  For any $\wh \in \Zh$ and any $(\vh,\underline{z}_h)\in [\Ub]^2$, using the non-dissipativity property \eqref{eq:non-dissipativity} followed by the linearity of $\th$ in its second and third argument, we get
  \[
  \begin{aligned}
    0 &= \th(\wh,\vh-\underline{z}_h,\vh - \underline{z}_h)\\
    & = \th(\wh,\vh,\vh - \underline{z}_h) - \th(\wh,\underline{z}_h,\vh - \underline{z}_h)\\
    & = - \th(\wh,\vh,\underline{z}_h) - \th(\wh,\underline{z}_h,\vh),
  \end{aligned}
  \]
  where we again used the non-dissipativity of $\th$ in the last passage.
\end{proof}

Property \eqref{eq:non-dissipativity} suggests that $\th$ can be interpreted as a centered discretization of convective terms. Deriving estimates that are quasi-robust for dominant convection requires strengthening the stability by a (generalized) upwinding term. To this end, given a family of strictly positive real numbers $\alpha \coloneqq (\alpha_T)_{T\in\Th}$, we define the convective stabilization bilinear form $j_{\alpha,h}:\U \times\U \to \Real$ such that, for all $(\wh,\underline{ {v}}_h)\in \U \times \U$,
\begin{equation*} 
  j_{\alpha,h}(\wh,\underline{ {v}}_h) \coloneqq \sum_{T\in \Th} j_{\alpha,T}(\underline{w}_T,\underline{ {v}}_T),
  \qquad
  j_{\alpha,T}(\underline{w}_T,\underline{ {v}}_T)\coloneqq \alpha_T \sum_{F \in \mathcal{F}_T} \int_F (w_F - w_T) \cdot ( {v}_F -  {v}_T).
\end{equation*}
The associated global and local seminorms are, respectively,
\begin{equation} \label{eq:j-seminorm}
  \text{%
    $\seminorm{\alpha,h}{\vh} \coloneqq  j_{\alpha,h}(\vh,\vh)^{\frac12}$
    and
    $\seminorm{\alpha,T}{\underline{v}_T} \coloneqq j_{\alpha,T}(\underline{v}_T,\underline{v}_T)^{\frac12}$
    for all $T \in \Th$.
  }
\end{equation}

\subsection{Discrete problem}

In what follows, to simplify the notation, we omit the dependence on time $t$ whenever it can be inferred from the context. Let two families $\beta \coloneqq (\beta_T)_{T \in \Th}$ and $\gamma \coloneqq (\gamma_T)_{T \in \Th}$ of functions of time
$\beta_T,\gamma_T \in C^0([0,\tF];\Real^+_*)$
be given. The semi-discrete counterpart of problem~\eqref{eq:weak} reads:
Find $(\underline{u}_h , \underline{b}_h,\ph,\underline{r}_h) \in [C^{1}([0,\tF];\Ub)]^2 \times [C^{0}(\lparen 0,\tF\rbrack;\Pb)]^2$ such that
\begin{subequations} \label{eq:discrete}
  \begin{equation}\label{eq:discrete:ic}
    \uh(0) = \Iu u_0,\qquad
    \bh(0) = \Iu b_0
  \end{equation}
  and, for all $t \in \lparen 0, \tF\rbrack$,
  \begin{alignat}{4} \label{discreteaRT3}
    \begin{aligned}[b]
      (\partial^{}_{t} \uh,\vh)_{0,h} + \nu a_h(\uh,\vh)  {}& +  \th(\uh,\uh,\vh) + j_{ {\beta},h}(\uh,\vh)\\
      {}& - \th(\bh,\bh, \vh) + B_h(\vh,\ph)
    \end{aligned}
    &= \int_\Omega f \cdot v_h
    &\quad& \forall  \vh \in  \Ub,
    \\ \label{discretebRT3}
    \begin{aligned}[b]
      (\partial^{}_{t} \bh, \wh)_{0,h} +\mu a_{h}(\bh, \wh) {}& + \th(\uh,\bh,\wh) + j_{\gamma,h}(\bh,\wh)
      \\
        {}& - \th(\bh,\uh, \wh) + B_h(\wh,\underline{r}_h)
    \end{aligned}
    &= \int_\Omega g \cdot w_h
    &\quad& \forall  \wh \in  \Ub,
    \\ \label{discrete:mass.gauge}
    B_h( \uh, \qh)
    = B_h( \bh, \underline{q}_h)
    &= 0
    &\quad& \forall \underline{q}_h \in  \Pb.
  \end{alignat}
\end{subequations}
Notice that an equivalent formulation without the pressures is classically obtained by taking the test functions in the divergence-free subspace~\eqref{eq:ZhZ} of $\Ub$, leading to the following problem:
Find $(\underline{u}_h , \underline{b}_h) \in [C^{1}([0,\tF];\Zh)]^2$ verifying~\eqref{eq:discrete:ic} such that, for all $(\vh, \wh) \in [ \Zh ]^2$,
\begin{subequations}\label{eq:discrete.bis}
  \begin{align}\label{eq:discrete.bis:momentum}
    (\partial^{}_{t} \uh,\vh)_{0,h}
    + \nu a_h(\uh,\vh)
    + \th(\uh,\uh,\vh)
    + j_{ {\beta},h}(\uh,\vh)
    - \th(\bh,\bh, \vh)
    &= \int_\Omega f \cdot v_h,
    \\ \label{eq:discrete.bis:maxwell}
    (\partial^{}_{t} \bh, \wh)_{0,h}
    + \mu a_{h}(\bh, \wh)
    + \th(\uh,\bh,\wh)
    + j_{\gamma,h}(\bh,\wh)
    - \th(\bh,\uh, \wh)
    &= \int_\Omega g \cdot w_h.
  \end{align}
\end{subequations}
The pressure $\underline{p}_h$ and the magnetic pressure $\underline{r}_h$ can then be recovered from~\eqref{discreteaRT3} and~\eqref{discretebRT3}, respectively, using test functions in $\Ub$; their uniqueness is guaranteed by the inf-sup condition~\eqref{eq:inf-sup}.

\subsection{Main results}\label{sec:discrete.problem:main.results}

We state here the main results of the analysis of the scheme, whose proofs are deferred to Section~\ref{sec:analysis}.
The natural energy norm on the discrete velocity and magnetic fields is the mapping $\tnorm{(\cdot,\cdot)}:[C^{0}([0,\tF];\Ub)]^2\to \Real$ defined, for all $(\vh,\wh) \in [C^{0}([0,\tF];\Ub)]^2$, by
\begin{equation}\label{eq:energy.norm}
  \begin{aligned}
    \tnorm{(\vh,\wh)}^2
    &\coloneqq
    \max_{t\in [0,\tF]} \| \vh(t) \|^2_{0,h}
    + \int^{\tF}_{0}\Big{(} \nu \| \vh(\tau) \|^2_{1,h} + | \vh(\tau) |^2_{ {\beta},h} \Big{)}
    \, d\tau
    \\
    &\quad
    + \max_{t\in [0,\tF]} \| \wh(t) \|^2_{0,h}
    +  \int^{\tF}_{0}\Big{(}  \mu \norm{1,h}{\wh(\tau)}^2   + | \wh(\tau) |^2_{\gamma,h} \Big{)} \, d\tau.
  \end{aligned}
\end{equation}

\begin{theorem}[Well-posedness of the scheme]\label{thm:well-posedness}
  There exists a unique solution to problem~\eqref{eq:discrete} (or, equivalently,~\eqref{eq:discrete.bis}), which additionally satisfies
  \begin{equation}\label{eq:energy_bound}
    \tnorm{(\uh,\bh)}^2 \lesssim
    e^{C\tF}\left[
      \norm{L^2(0,\tF;L^2(\Omega)^d)}{f}^2
      + \norm{L^2(0,\tF;L^2(\Omega)^d)}{g}^2
      + \seminorm{H^1(\Omega)^d}{u_0}^2
      + \seminorm{H^1(\Omega)^d}{b_0}^2
      \right].
  \end{equation}
  Here, $C>0$ depends only on the mesh regularity parameter and the polynomial degree $k$; in particular, it is independent of $h$, $\nu$, $\mu$, $\beta$, and $\gamma$.
\end{theorem}

\begin{proof}
  See Section~\ref{sec:analysis:well-posedness}.
\end{proof}

The error analysis is carried out in the spirit of \cite{Di-Pietro.Droniou:18}, and the regularity requirements on the exact solution will be formulated in terms of the broken Sobolev spaces
\[
W^{s,p}(\Th) \coloneqq \left\{
q \in L^p(\Omega) \st \text{%
  $q_{|T} \in W^{s,p}(T)$ for all $T \in \Th$
}
\right\}
\]
as well as the broken Hilbert spaces $H^s(\Th) \coloneqq W^{s,2}(\Th)$.

Recalling Remark~\ref{rem:interpolate.div.free.functions}, we define the errors on the velocity and magnetic field, respectively, as
\begin{equation}\label{eq:approximation.errors}
  \eu \coloneqq \uh - \Iu u\in\Zh,\qquad
  \eb \coloneqq \bh - \Iu b\in\Zh.
\end{equation}
In Theorem~\ref{thm:error.estimate} below, we establish a basic estimate where the quantity $\tnorm{(\eu,\eb)}$ is bounded in terms of the following consistency errors:
\begin{itemize}
\item The \emph{time consistency error} which, given $w \in C^1([0,\tF]; H^1(\Omega)^d)$, is such that
  \begin{equation} \label{eq:E.time.h}
    \Err{\rm time}(w,\underline{z}_h)
    \coloneqq \int_\Omega\partial_t w \cdot z_h
    - (\partial^{}_{t} \Iu w,\underline{z}_h)_{0,h}
    \qquad \forall \underline{z}_h \in \Ub;
  \end{equation}

\item The \emph{diffusive consistency error} which, given $w \in H_0^1(\Omega)^d \cap H^2(\Th)^d$, is such that
  \begin{equation}\label{eq:E.diff.h}
    \Err{\rm diff}(w, \underline{z}_h)
    \coloneqq
    - \int_\Omega \Delta w \cdot z_h
    - a_h(\Iu w,\underline{z}_h)
    \qquad \forall \underline{z}_h \in \Ub;
  \end{equation}

\item The \emph{convective consistency error} such that, for all $(w,v) \in H_0^1(\Omega)^d \times H_0^1(\Omega)^d$,
  \begin{equation}\label{eq:E.conv.h}
    \Err{\rm conv}(w,v,\underline{z}_h) \coloneqq  \int_\Omega (w \cdot  \grad) v \cdot  {z}_h
    - \th(\Iu w,\Iu v, \underline{z}_h)
    \qquad \forall \underline{z}_h \in \Ub.
  \end{equation}
\end{itemize}

\begin{theorem}[Basic error estimate]\label{thm:error.estimate}
  Let the solution of \eqref{eq:weak} be such that $u,b \in H^{1}(0,\tF; H^{1}_{0}(\Omega)^d) \cap L^{1}(0,\tF; W^{1,\infty}(\Omega)^d) \cap L^2(0,\tF; H^2(\Th)^d)$ and $p,r \in L^2(0,\tF;L^2_0(\Omega) \cap H^1(\Omega))$.
  Define the global consistency error, for all $(\vh,\wh) \in \Ub \times \Ub$, by
  \begin{equation}\label{eq:Eh}
    \begin{aligned}
      \mathcal{E}_h(\vh,\wh)
      &\coloneqq
      \Err{\rm time}(u,\vh)
      + \Err{\rm time}(b,\wh)
      + \nu \Err{\rm diff}(u,\vh) + \mu\Err{\rm diff}(b,\wh)
      \\
      &\quad
      + \Err{\rm conv}(u,u,\vh)
      - j_{ {\beta},h}(\IUh u,\vh)
      - \Err{\rm conv}(b,b,\vh)
      \\
      &\quad
      + \Err{\rm conv}(u,b,\wh)
      - j_{ {\gamma},h}(\IUh b,\wh)
      - \Err{\rm conv}(b,u,\wh).
    \end{aligned}
  \end{equation}
  Notice that, as $u$ and $b$ are functions of time, $\mathcal{E}_h(\vh,\wh)$ is a function of time as well.
  Then, recalling the definition~\eqref{eq:energy.norm} of the energy norm, we have
  \begin{equation}\label{eq:basic.estimate}
    \tnorm{(\eu,\eb)}^2
    \lesssim K(u,b) \int^{\tF}_{0}
    \left| \mathcal{E}_h(\eu(\tau),\eb(\tau)) \right|
    \, d\tau,
  \end{equation}
  where
  $K(u,b) \coloneqq \exp\left( C \int^{\tF}_{0} \big( | u(\tau) |_{ {W}^{1,\infty}(\Omega)^d} + | b(\tau) |_{ {W}^{1,\infty}(\Omega)^d} \big) \, d\tau \right)$
  for a real number $C>0$ that only depends on the mesh regularity parameter and on the polynomial degree $k$.
\end{theorem}

\begin{proof}
  See Section~\ref{sec:analysis:error.estimate}.
\end{proof}

For all $\eta>0$,
all $w \in L^{\infty}(\Omega)^d$,
all $\alpha \coloneqq (\alpha_T)_{T \in \Th}$ strictly positive real numbers,
and all $T \in \Th$, we define the local P\'eclet number
\begin{equation}\label{eq:PeT}
  \Pe_T(\eta,w,\alpha_T)
  \coloneqq \frac{\left(\alpha_T + \| w \|_{ L^{\infty}(T)^d}\right) h_T}{\eta}
\end{equation}
and the real number
\begin{equation}\label{eq:chi.alpha}
  \chi(\eta,w,\alpha)
  \coloneqq \max_{T \in \Th,\, \Pe_T(\eta,w,\alpha_T)>1} \frac{\norm{L^\infty(T)^d}{w}}{\alpha_T},
\end{equation}
with the convention that $\chi(\eta,w,\alpha) = 0$ if the set over which the maximum is taken is empty.

The convergence rates stated in the following corollary are obtained by estimating the consistency errors that appear in the right-hand side of \eqref{eq:Eh} for smooth-enough solutions.

\begin{corollary}[Convergence rates for smooth solutions]\label{cor:convergence.rates}
  Let the assumptions and notations of Theorem~\ref{thm:error.estimate} hold,
  further assume $u, b \in L^2(0,\tF; H^{k+2}(\mathcal{T}_h)^d) \cap C^0([0,\tF]; W^{1,\infty}(\Omega)^d)$
  and $ \frac{d u}{dt}, \frac{db}{dt} \in L^2(0,\tF; H^{k+1}(\mathcal{T}_h)^d)$,
  and let
  \begin{align}
    \mathfrak{P}_T &\coloneqq
      \max\Big\{
      \norm{L^\infty(0,\tF)}{\Pe_T(\nu, u, \beta_T)},
      \norm{L^\infty(0,\tF)}{\Pe_T(\mu, u, \gamma_T)},\nonumber\\
      &\qquad\qquad\qquad\norm{L^\infty(0,\tF)}{\Pe_T(\nu, b, \beta_T)},
      \norm{L^\infty(0,\tF)}{\Pe_T(\mu, b, \gamma_T)}
      \Big\},
      \label{eq:frak.PT}
    \\ \nonumber
    \Theta &\coloneqq \max\left\{
    \norm{L^{\infty}(0,\tF)}{\chi(\nu,u,\beta)},
    \norm{L^{\infty}(0,\tF)}{\chi(\nu,b,\beta)},
    \norm{L^{\infty}(0,\tF)}{\chi(\mu,u,\gamma)},
    \norm{L^{\infty}(0,\tF)}{\chi(\mu,b,\gamma)}
    \right\}.
  \end{align}
  The continuity and strict positivity of $\beta_T$ and $\gamma_T$ on $[0,\tF]$, together with the assumed continuity of $u$ and $b$ in $W^{1,\infty}(\Omega)^d$, ensure that $\mathfrak{P}_T$ and $\Theta$ are finite.
  Then, with a hidden constant independent of $h$, $\nu$, $\mu$, $\beta$, and $\gamma$,
  \begin{equation}\label{eq:convergence.rate}
    \tnorm{(\eu,\eb)}^2
    \lesssim K(u,b) \sum_{T \in \Th} \left[
      h_T^{2(k+1)} \, \mathcal{N}_1(T)
      + h_T^{2k+1} \min(1,\mathfrak{P}_T) \, \mathcal{N}_2(T)
      \right],
  \end{equation}
  where
  \[
  \begin{aligned}
    \mathcal{N}_1(T)
    &\coloneqq
    \int^{\tF}_{0} \left(
    \Seminorm{H^{k+1}(T)^d}{\frac{d u}{dt}(t)}^2
    + \Seminorm{H^{k+1}(T)^d}{\frac{d b}{dt}(t)}^{2}
    \right) dt
    \\
    &\quad
    +\int^{\tF}_{0} \left(
    \nu \seminorm{H^{k+2}(T)^d}{u(t)}^2
    + \mu \seminorm{H^{k+2}(T)^d}{b(t)}^2
    \right) dt
    \\
    &\quad
    +\int^{\tF}_{0}
    \left(
    \seminorm{W^{1,\infty}(T)^d}{u(t)}^2
    + \seminorm{W^{1,\infty}(T)^d}{b(t)}^2
    \right)
    \left(
    \seminorm{H^{k+1}(T)^d}{u(t)}^2
    + \seminorm{H^{k+1}(T)^d}{b(t)}^2
    \right) dt,
    \\
    \mathcal{N}_2(T)
    &\coloneqq
    (1 + \Theta) \int^{\tF}_{0}
    \begin{aligned}[t]
      &\left(
      \norm{L^{\infty}(T)^d}{u(t)}
      + \beta_T(t)
      + \norm{L^{\infty}(T)^d}{b(t)}
      + \gamma_T(t)
      \right)
      \\
      &\times
      \left(
      \seminorm{H^{k+1}(T)^d}{u(t)}^2
      + \seminorm{H^{k+1}(T)^d}{b(t)}^2
      \right) dt.
    \end{aligned}
  \end{aligned}
  \]
\end{corollary}

\begin{proof}
  See Section~\ref{sec:analysis:convergence.rates}.
\end{proof}

\begin{remark}[Reynolds- and Hartmann-semi-robustness]
 The estimate~\eqref{eq:convergence.rate} is robust with respect to all the dimensionless numbers entering the definition~\eqref{eq:frak.PT} of $\mathfrak{P}_T$. The two quantities involving $u$ are Reynolds-type numbers, whereas the two quantities involving $b$ jointly characterize a Hartmann-type regime. Indeed, when the stabilization contributions are negligible, the geometric mean of the latter two reduces to $\norm{L^\infty(T)^d}{b}h_T/\sqrt{\nu\mu}$, which has the form of a local Hartmann number.
\end{remark}

\begin{remark}[Convergence rate] \label{rm:convergence_rate}
  In~\eqref{eq:convergence.rate}, the local contribution to the error from a mesh element $T \in \Th$ is of order $h_T^{k+\frac12}$ if $\mathfrak{P}_T \ge 1$,
  $h_T^{k+1}$ if $\mathfrak{P}_T \le h_T$,
  while intermediate orders of convergence are observed when $h_T < \mathfrak{P}_T < 1$.
  This means that, in the (pre-asymptotic) convection-dominated regime, an order of convergence of $h^{k+\frac12}$ is expected, and that this order increases up to $h^{k+1}$ in the (asymptotic) diffusion-dominated regime.
  Unlike other methods, the accuracy therefore increases as the mesh is refined, since the discretization of the diffusive terms is of higher-order than that of the convective terms.
\end{remark}

\begin{remark}[Non-convex domains]
  When dealing with non-convex domains, the embedding $H(\operatorname{curl};\Omega) \cap H_0(\operatorname{div};\Omega) \hookrightarrow H^1(\Omega)^3$ no longer holds, and the solution of the weak formulation \eqref{eq:weak} does not, in general, solve the strong problem \eqref{eq:strong} almost everywhere.
  As a result, while the scheme remains stable, one may observe convergence to solutions that do not make sense from a physical standpoint.
  An example of this phenomenon is provided in \cite[Chapter 1]{Arnold:18}, where the eigenvalue problem for the vector Laplacian with magnetic boundary conditions is considered.
  To restore convergence to the correct solution, one could, e.g., seek the magnetic field $b$ in an $H(\operatorname{curl}; \Omega)$-conforming (e.g., N\'ed\'elec-type) finite element space.
  Achieving Reynolds- and Hartmann-robustness in this case does not appear to be straightforward, particularly because property \eqref{thbound} is not easy to replicate with curl-based formulations of the convective term.
  This subject is currently under investigation and will be the subject of future work.
\end{remark}


\section{Numerical results}\label{sec:numerical.results}

In this section, we numerically verify the theoretical results.
We test the method on successively refined simplicial mesh families and with polynomial degrees $k \in \left\{0,1,2\right\}$.
Our implementation is based on the \texttt{HArDCore} library\footnote{\url{https://github.com/jdroniou/HArDCore}} and makes extensive use of the linear algebra \texttt{Eigen} open-source library~\cite{Guennebaud.Jacob.ea:10}.
In all the numerical results of this section, we have fixed $\tF=1$. The time discretization is performed using an implicit Crank--Nicolson method, with the number of time steps set according to the following formula:
\begin{equation*}
  N_{\tF} = \max\left\{ 10, \left\lceil h^{-\frac{k+1}{2}} \right\rceil \right\},
\end{equation*}
where the ceiling operator is defined by $\lceil x\rceil \coloneqq \min \left\{ n \in \Relative \st n \ge x \right\}$ for all $x \in \Real$.
This choice ensures that, asymptotically, the global time-integration error associated with the time-step size $\Delta t = \frac{\tF}{N_{\tF}}$ is of order $h^{k+1}$.
At each time step, the nonlinear algebraic problem is solved using Newton's method until the initial residual has been reduced by eight orders of magnitude.
Linear systems are solved using the direct solver \texttt{Pardiso}~\cite{Schenk.Gartner.ea:01}.
Due to their linearity, the constraints \eqref{discrete:mass.gauge} hold at each nonlinear iteration, preserving the pointwise divergence-free property of the solution.

For the tests involving nonhomogeneous Dirichlet conditions, the boundary face unknowns are fixed to the corresponding interpolated boundary data. This is the standard extension of the scheme to nonhomogeneous data; the analysis above is restricted to homogeneous conditions only to simplify the presentation.

Denote by $t^n \coloneqq n \Delta t$ the time after $n$ steps.
The two families of upwind coefficients $\beta$ and $\gamma$ are obtained at each time step $n \ge 1$ by setting, for all $T \in \Th$,
\begin{equation}\label{eq:C_stab_upwind}
  \beta_T(t^n) = \gamma_T(t^n)
  \coloneqq
  C_{\rm stab} \max\left\{
  10^{-4}, \norm{L^\infty(T)^d}{u_T(t^{n-1})} + \norm{L^\infty(T)^d}{b_T(t^{n-1})}
  \right\}.
\end{equation}
The user-dependent parameter $C_{\rm stab}$ influences the amount of upwinding in the scheme.
Too much upwinding might prevent the scheme from reaching the asymptotic behavior, whereas too little upwinding may render the method unstable.
Notice, in passing, that our regime-dependent analysis justifies the classical notion that convective stabilization is essential in the convection-dominated regime, while it can be reduced or even turned off in the diffusion-dominated regime.
We have also observed that increasing the polynomial degree tends to improve the stability of the scheme, and therefore allows for smaller values of $C_{\rm stab}$.

In accordance with Remark~\ref{rm:convergence_rate}, we expect to observe convergence in $h^{k+\frac12}$ (pre-asymptotic rate) for small values of the diffusion coefficients and/or coarser meshes, with an improvement up to $h^{k+1}$ (asymptotic rate) for larger values of the diffusion coefficients and/or fine meshes.
Notice, however, that, in practice, achieving the asymptotic rate may require extremely fine meshes, which is not always feasible.

\subsection{Two-dimensional convergence test case}\label{sec:numerical.results:2d}

We take $\Omega=(0,1)^2$ and select the forcing terms $f$ and $g$ so that the exact solution is, for all $(x,y) \in \Omega$ and all $t \in [0,1]$,
\[
\begin{gathered}
  u(x,y,t)= -e^{-0.5 t}\begin{bmatrix}\sin (2\pi x) \sin (2\pi y)\\ \cos (2\pi x) \cos (2\pi y)\end{bmatrix},\quad
  b(x,y,t)= -e^{-0.5 t}\begin{bmatrix}\cos (2\pi x) \cos (2\pi y)\\ \sin (2\pi x) \sin (2\pi y)\end{bmatrix},\\
  p(x,y,t)= e^{-0.5 t}(\sin (2\pi x) \cos (2\pi y)),\quad
  r(x,y,t)= e^{-0.5 t}(\cos (2\pi x) \sin (2\pi y)).
\end{gathered}
\]
To capture the different regimes, we have used $\nu,\mu\in \{1, 10^{-3}, 10^{-6}\}$.
The parameter $C_{\rm stab}$ present in \eqref{eq:C_stab_upwind} has been set to 1.
Figure~\ref{fig:manufactured_solution_one} depicts the energy error of the discrete solution $\tnorm{(\eu,\eb)}$ as a function of $h$.
For each pair of consecutive meshes, the corresponding convergence rate is reported above the curve.
In Figure~\ref{fig:manufactured_solution_one.a}, an asymptotic behavior with convergence rates close to $k+1$ is observed for all values $k$.
In Figure~\ref{fig:manufactured_solution_one.i}, on the other hand the behavior is pre-asymptotic, with convergence rate closer to $k + \frac12$.
Figures~\ref{fig:manufactured_solution_one.a},~\ref{fig:manufactured_solution_one.b},~\ref{fig:manufactured_solution_one.c},~\ref{fig:manufactured_solution_one.d},~\ref{fig:manufactured_solution_one.g} show that, for the considered configuration, the asymptotic behavior tends to be achieved quicker with higher polynomial degrees.
The transition from the pre-asymptotic to the asymptotic behavior can be observed for $k=0$ in Figures~\ref{fig:manufactured_solution_one.e},~\ref{fig:manufactured_solution_one.f},~\ref{fig:manufactured_solution_one.h}, ~\ref{fig:manufactured_solution_one.i}.

\begin{sidewaysfigure}
  \centering
  \subcaptionbox{$(\nu,\mu) = (1,1)\label{fig:manufactured_solution_one.a}$}[0.225\textwidth]{%
  \centering
  \includegraphics[width=\linewidth]{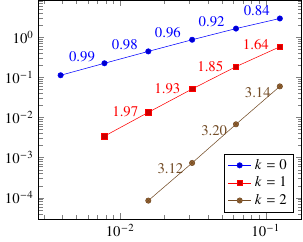}
}\hspace{0.01\textwidth}
  \subcaptionbox{$(\nu,\mu) = (1,10^{-3})\label{fig:manufactured_solution_one.b}$}[0.225\textwidth]{%
  \centering
  \includegraphics[width=\linewidth]{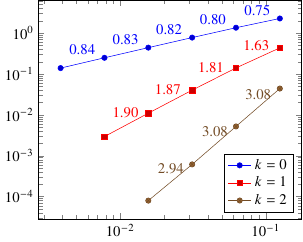}
}\hspace{0.01\textwidth}
  \subcaptionbox{$(\nu,\mu) = (1,10^{-6})\label{fig:manufactured_solution_one.c}$}[0.225\textwidth]{%
  \centering
  \includegraphics[width=\linewidth]{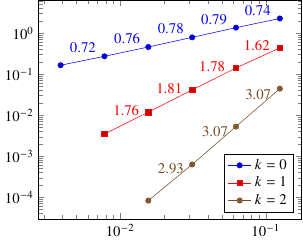}
}
  \medskip\\
  \subcaptionbox{$(\nu,\mu) = (10^{-3},1)\label{fig:manufactured_solution_one.d}$}[0.225\textwidth]{%
  \centering
  \includegraphics[width=\linewidth]{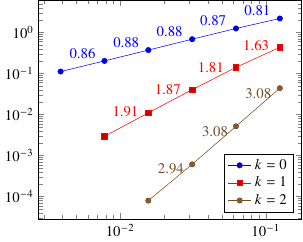}
}\hspace{0.01\textwidth}
  \subcaptionbox{$(\nu,\mu) = (10^{-3},10^{-3})\label{fig:manufactured_solution_one.e}$}[0.225\textwidth]{%
  \centering
  \includegraphics[width=\linewidth]{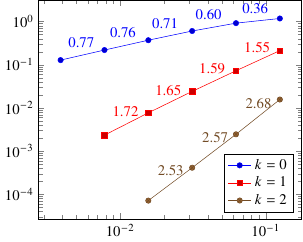}
}\hspace{0.01\textwidth}
  \subcaptionbox{$(\nu,\mu) = (10^{-3},10^{-6})\label{fig:manufactured_solution_one.f}$}[0.225\textwidth]{%
  \centering
  \includegraphics[width=\linewidth]{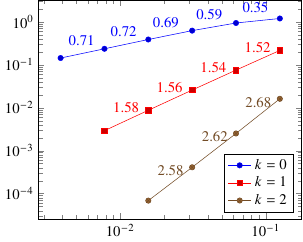}
}
  \medskip\\  
  \subcaptionbox{$(\nu,\mu) = (10^{-6},1)\label{fig:manufactured_solution_one.g}$}[0.225\textwidth]{%
  \centering
  \includegraphics[width=\linewidth]{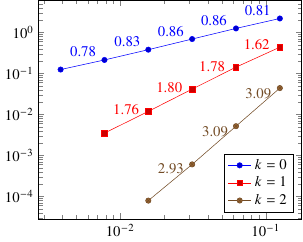}
}\hspace{0.01\textwidth}
  \subcaptionbox{$(\nu,\mu) = (10^{-6},10^{-3})\label{fig:manufactured_solution_one.h}$}[0.225\textwidth]{%
  \centering
  \includegraphics[width=\linewidth]{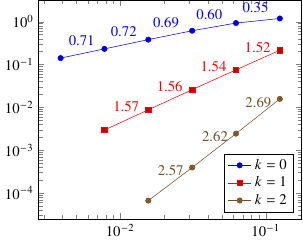}
}\hspace{0.01\textwidth}
  \subcaptionbox{$(\nu,\mu) = (10^{-6},10^{-6})\label{fig:manufactured_solution_one.i}$}[0.225\textwidth]{%
  \centering
  \includegraphics[width=\linewidth]{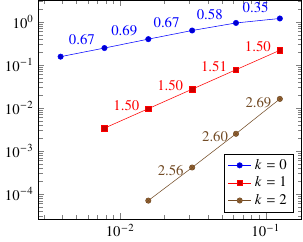}
}
  \caption{Error $\tnorm{(\eu,\eb)}$ as a function of the meshsize $h$ for the two-dimensional test case of Section~\ref{sec:numerical.results:2d}.}
  \label{fig:manufactured_solution_one}
\end{sidewaysfigure}

In Table~\ref{tab:2D DOFs}, the number of degrees of freedom (DOFs) required for the previous computations is reported as a function of the polynomial degree $k$ and the mesh size $h$.
\begin{table}\centering  
  \begin{tabular}{cccc}
    \toprule
    $h$ & $k=0$ & $k=1$ & $k=2$ \\
    \midrule
    \num[exponent-mode = scientific]{0.125}      & \num{2434}    & \num{4418}    & \num{6402} \\
    \num[exponent-mode = scientific]{0.0625}     & \num{9794}    & \num{17794}   & \num{25794} \\
    \num[exponent-mode = scientific]{0.03125}    & \num{39298}   & \num{71426}   & \num{103554} \\
    \num[exponent-mode = scientific]{0.015625}   & \num{157442}  & \num{286210}  & \num{414978} \\
    \num[exponent-mode = scientific]{0.0078125}  & \num{630274}  & \num{1145858} &  -- \\
    \num[exponent-mode = scientific]{0.00390625} & \num{2522114} & --      & -- \\
    \bottomrule
  \end{tabular}
    \caption{Number of DOFs as a function of $k$ and $h$}
    \label{tab:2D DOFs}
\end{table}
Moreover, Figure~\ref{fig:manufactured_solution_one_DOF} shows that, at a comparable number of DOFs, increasing the polynomial degree yields a smaller error than refining the mesh at lower degree.
\begin{figure}
  \centering
\subcaptionbox{$(\nu,\mu) = (1,10^{-3})$}{%
  \centering
  \includegraphics[width=0.4\textwidth]{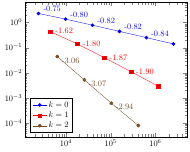}
}
  \hfil
\subcaptionbox{$(\nu,\mu) = (10^{-6},10^{-3})$}{%
  \centering
  \includegraphics[width=0.4\textwidth]{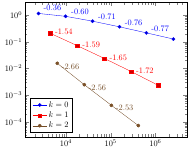}
}
\caption{Error $\tnorm{(\eu,\eb)}$ as a function of the square root of the number of DOFs.}
\label{fig:manufactured_solution_one_DOF}
\end{figure}

\subsection{Three-dimensional convergence test case}\label{sec:numerical.results:3d}

We take $\Omega=(0,1)^3$ and select the forcing terms $f$ and $g$ so that the exact solution is, for all $(x,y,z) \in \Omega$ and all $t \in [0,1]$,
\[
\begin{gathered}
  u= e^{-0.5 t}\begin{bmatrix}\sin (2\pi x) \cos (2\pi y) \cos (2\pi z)\\ -\cos (2\pi x) \sin (2\pi y) \cos (2\pi z)\\0
  \end{bmatrix},\quad
  b= e^{-0.5 t}\begin{bmatrix}\cos (2\pi x) \sin (2\pi y) \sin (2\pi z)\\ \sin (2\pi x) \cos (2\pi y) \sin (2\pi z)\\ -2\sin (2\pi x) \sin (2\pi y) \cos (2\pi z)\end{bmatrix},
  \\
  p= e^{-0.5 t}(\sin (2\pi x) \sin (2\pi y) \cos (2\pi z)),\qquad
  r= e^{-0.5 t}(\cos (2\pi x) \cos (2\pi y) \sin (2\pi z)).
\end{gathered}
\]
We set $\nu=1$ and consider $\mu\in\{1,10^{-2}\}$. For each value of $\mu$, multiple computations are performed by varying the upwind stabilization parameter $C_{\rm stab}$ in \eqref{eq:C_stab_upwind}. Specifically, $C_{\rm stab}$ is set to $0$ (i.e., no convective stabilization) and $1$ for $\mu=1$, and to $0.5$ and $1$ for $\mu=10^{-2}$. Figure~\ref{fig:manufactured_solution_two} shows the energy error $\tnorm{(\eu,\eb)}$ as a function of $h$. Continuous lines correspond to $C_{\rm stab}=0$ and $C_{\rm stab}=0.5$, while dotted lines correspond to $C_{\rm stab}=1$. As in the 2D case, both the diffusion-dominated regime (see Figure~\ref{fig:manufactured_solution_two.a}) and the convection-dominated regime (see Figure~\ref{fig:manufactured_solution_two.b}) are observed.

The computations with $C_{\rm stab}=0$ fall outside the strict-positivity assumption on the stabilization coefficients used in the analysis (in particular, $\chi$ is then undefined). The well-posedness proof nevertheless extends to the centered scheme: the corresponding stabilization seminorm simply vanishes, while the energy estimate remains valid.

\begin{figure}
  \centering
  \subcaptionbox{$(\nu,\mu) = (1,1)\label{fig:manufactured_solution_two.a}$}[0.42\textwidth]{%
  \centering
  \includegraphics[width=\linewidth]{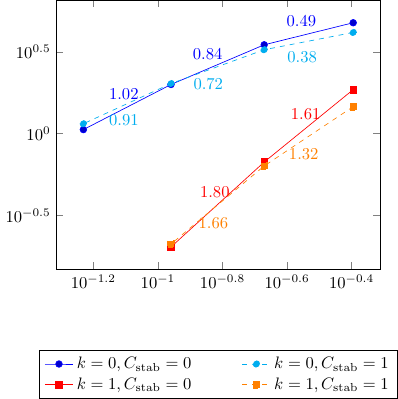}
}
  \hfil
\subcaptionbox{$(\nu,\mu) = (1,10^{-2})\label{fig:manufactured_solution_two.b}$}[0.42\textwidth]{%
  \centering
  \includegraphics[width=\linewidth]{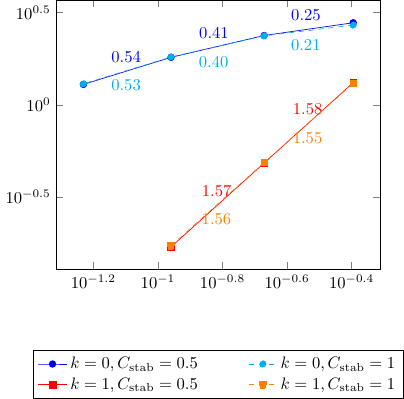}
}
  \caption{Error $\tnorm{(\eu,\eb)}$ as a function of the meshsize $h$ for the three-dimensional test case of Section~\ref{sec:numerical.results:3d}.}
  \label{fig:manufactured_solution_two}
\end{figure}

\begin{remark}[Choice of $C_{\rm stab}$] The numerical experiments show that the choice $C_{\rm stab} = 1$ in the three-dimensional case can result in excessive numerical dissipation on refined meshes.
    Moreover, the choice of $C_{\rm stab}$ has an influence on the local Reynolds and Hartmann numbers, and can therefore affect not only the actual value of the error, but also the observed convergence rate (values too large can result in pre-asymptotic instead of asymptotic rates).
\end{remark}

\subsection{Magnetic lid-driven cavity}\label{sec:numerical.tests:cavity}

In this section we present a more physical test case corresponding to the magnetic lid-driven cavity, inspired by the three-dimensional case of \cite{Ben-Salah.Soulaimani.ea:01}.
Our current three-dimensional implementation is not parallel and therefore does not permit large-scale computations. For this reason, we opted for a two-dimensional version of the test case. This reduction cannot be compared directly with the three-dimensional benchmark of~\cite{Ben-Salah.Soulaimani.ea:01}, but it remains a meaningful qualitative MHD test. As shown below, the numerical results exhibit distinctive features, such as the formation of a vortex-like structure in the magnetic field induced by the recirculating flow within the lid-driven cavity.
The domain $\Omega$ is the unit square $(0,1)^2$, where the upper boundary represents a lid moving with constant horizontal velocity component $U = 1$.
The fluid viscosity is $\nu = 10^{-3}$, corresponding to a Reynolds number of $10^3$.
The initial conditions are $u_0=0$ inside the cavity and $b_0 = B e_2$, with $\{e_1, e_2\}$ unit vectors along the $x$- and $y$-axis.
The boundary conditions match $u_0$ and $b_0$, except for the velocity at the top of the cavity which is imposed as $u=Ue_1$.
Our goal here is to investigate the effect on both the velocity and magnetic fields of varying the value of the magnetic diffusivity $\mu$.

\begin{figure}
  \centering
  \includegraphics[width=0.35\textwidth]{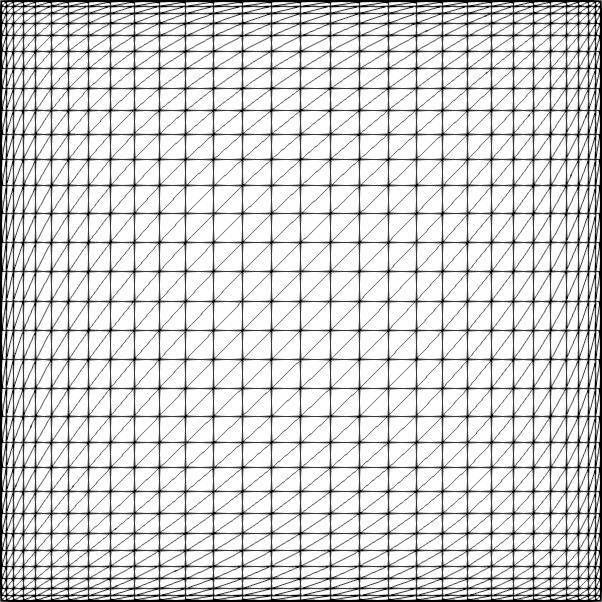}
  \caption{Mesh used in the numerical test of Section \ref{sec:numerical.tests:cavity}.}
  \label{fg:mesh_LDC}
\end{figure}
Figure~\ref{fg:mesh_LDC} depicts the computational mesh used in this section, which contains 2048 elements and is refined towards the boundary.
For the space discretization, we used the polynomial degree $k=2$.
The selected upwind stabilization parameter is set to $C_{\rm stab}=1$.
For the time integration, we employed the implicit Euler method for the first 10 time steps, followed by the Crank--Nicolson scheme for all subsequent time steps.
This choice facilitates the convergence to a steady-state solution (which could, otherwise, be affected by oscillations in time as pointed out in \cite{Luskin01011982, Rannacher1984}).
The simulations were run up to a final time of $\tF=500$ using a total of 3000 time steps.

Imposing $B=0$ leads to solving the hydrodynamic limit, where $\mu$ has no influence. The corresponding velocity is displayed In Figure \ref{fig:U1-B0-k2-nu0.001}, in which the classical recirculating vortices are seen at bottom corners.
\begin{figure}
  \centering
  \includegraphics[height=0.35\textwidth, valign=t]{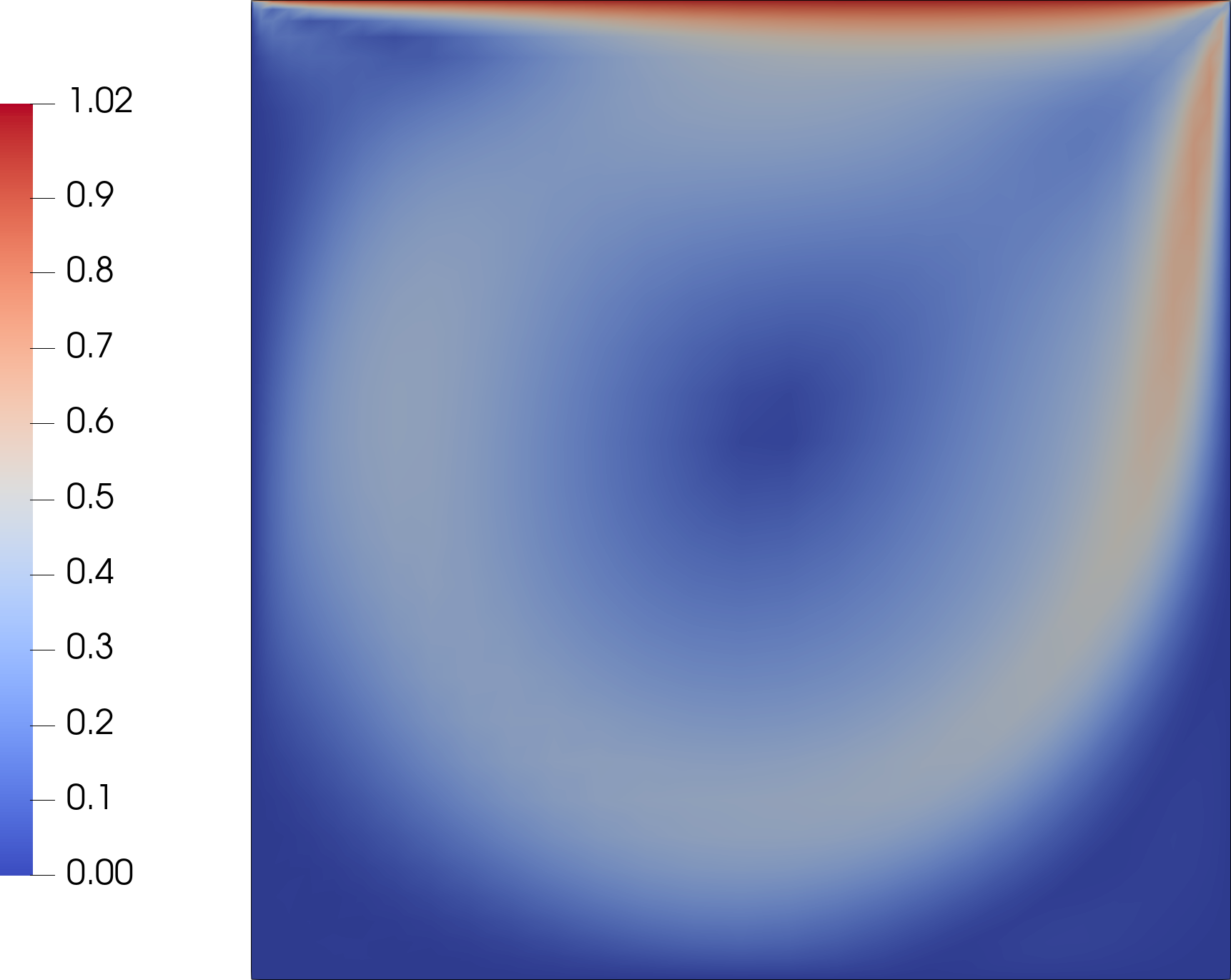}
  \hspace{1cm}
  \includegraphics[height=0.35\textwidth, valign=t]{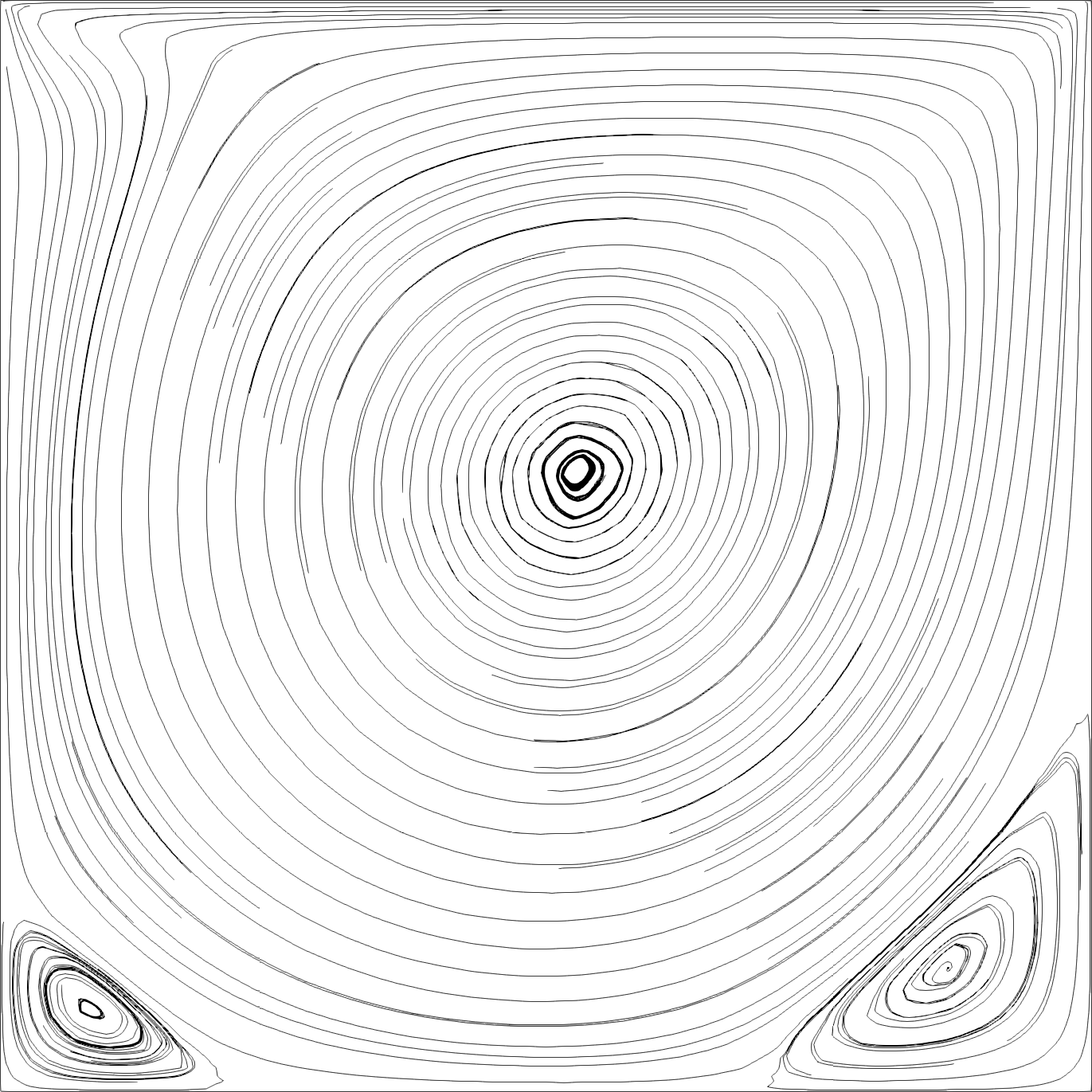}
  \caption{Magnitude of the velocity field (left) and streamlines (right) for $B=0$.}
  \label{fig:U1-B0-k2-nu0.001}
\end{figure}

Let us now fix $B=0.1$ and consider the behavior of the solution when the magnetic diffusivity $\mu$ varies.
\begin{figure}
  \centering
  \subcaptionbox{$B=0.1$ and $\mu=1$}[1\textwidth]{
    \includegraphics[height=0.35\textwidth, valign=t]{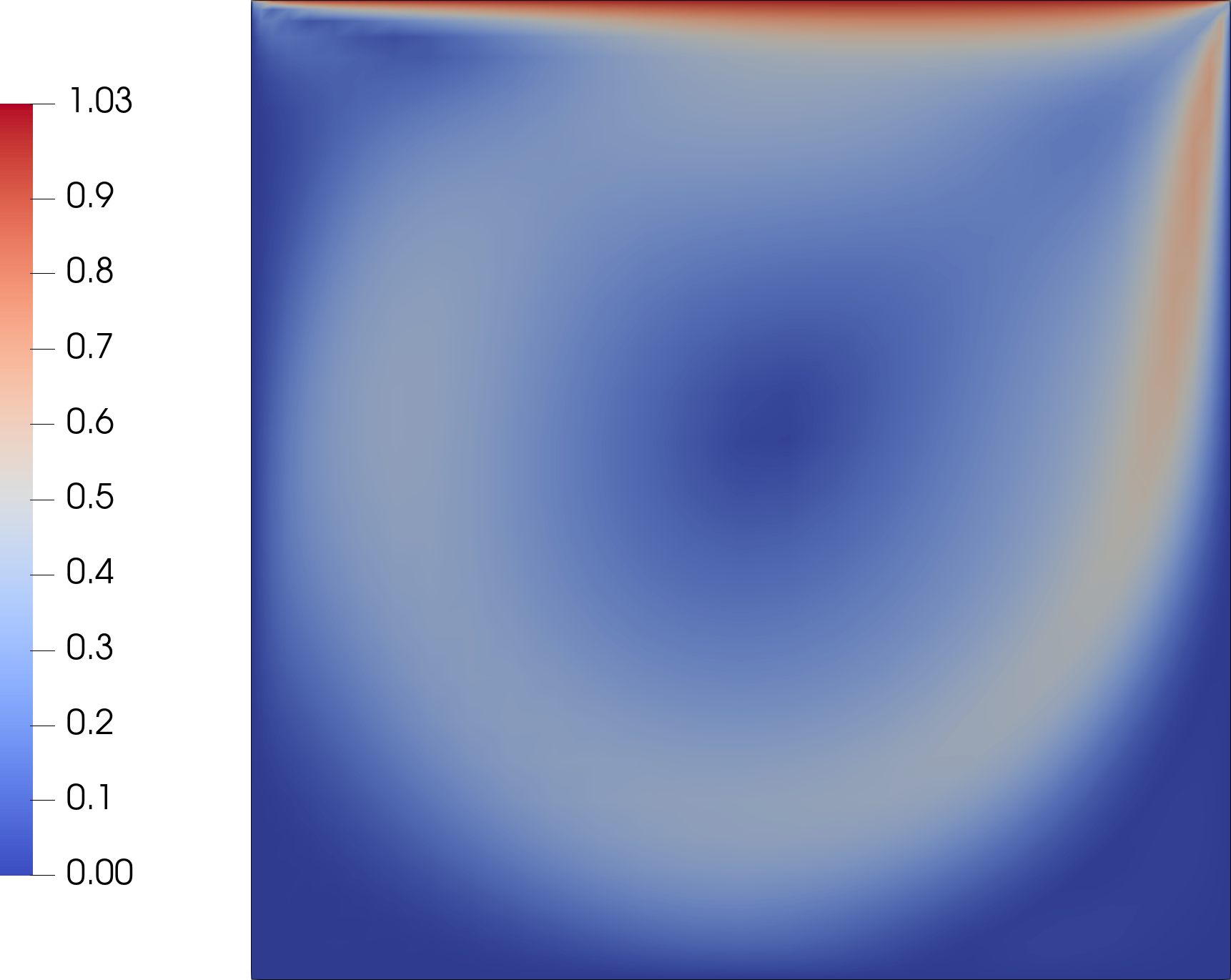}
    \hspace{1cm}
    \includegraphics[height=0.35\textwidth, valign=t]{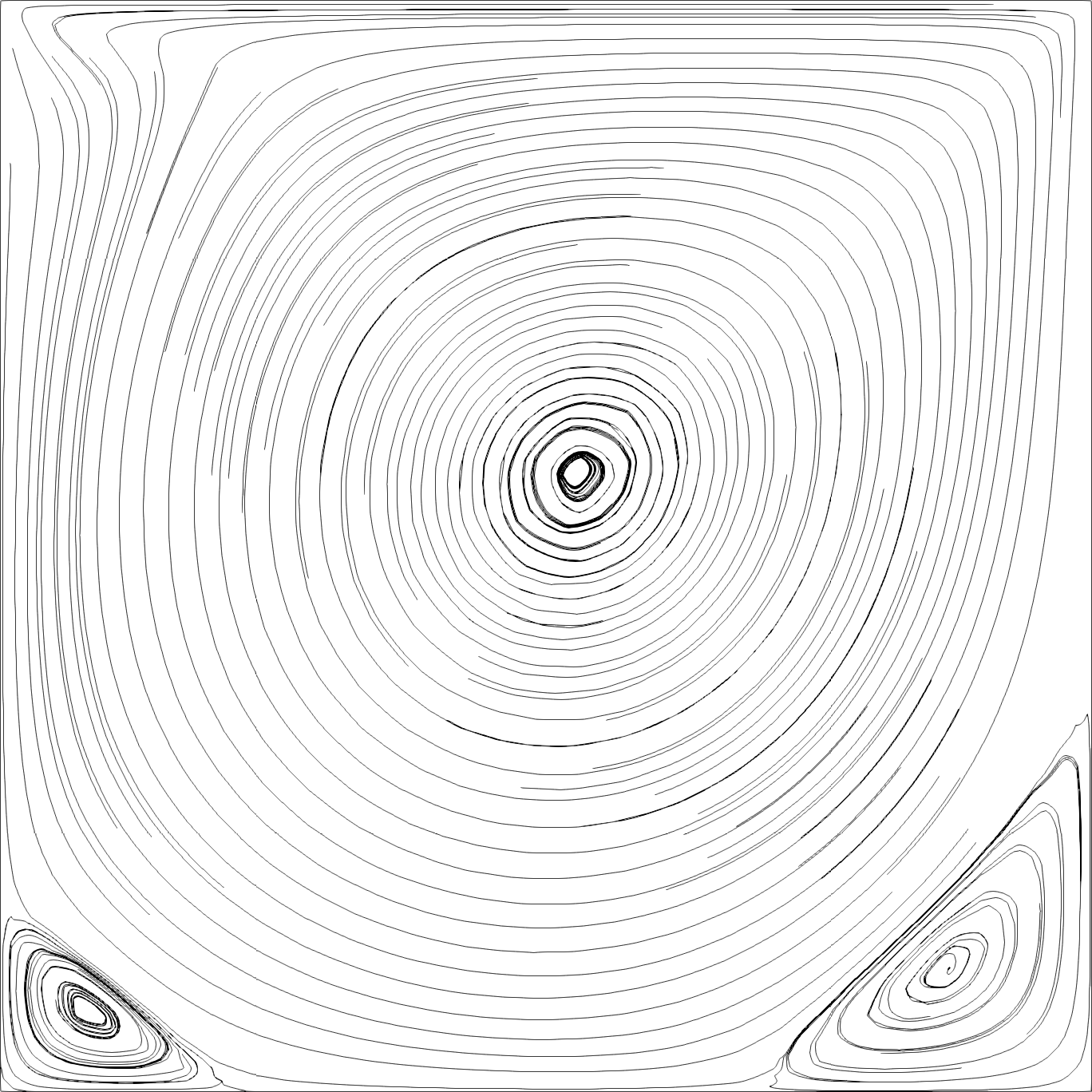}
  }
  \medskip\\
  \subcaptionbox{$B=0.1$ and $\mu=10^{-3}$}[1\textwidth]{
    \includegraphics[height=0.35\textwidth, valign=t]{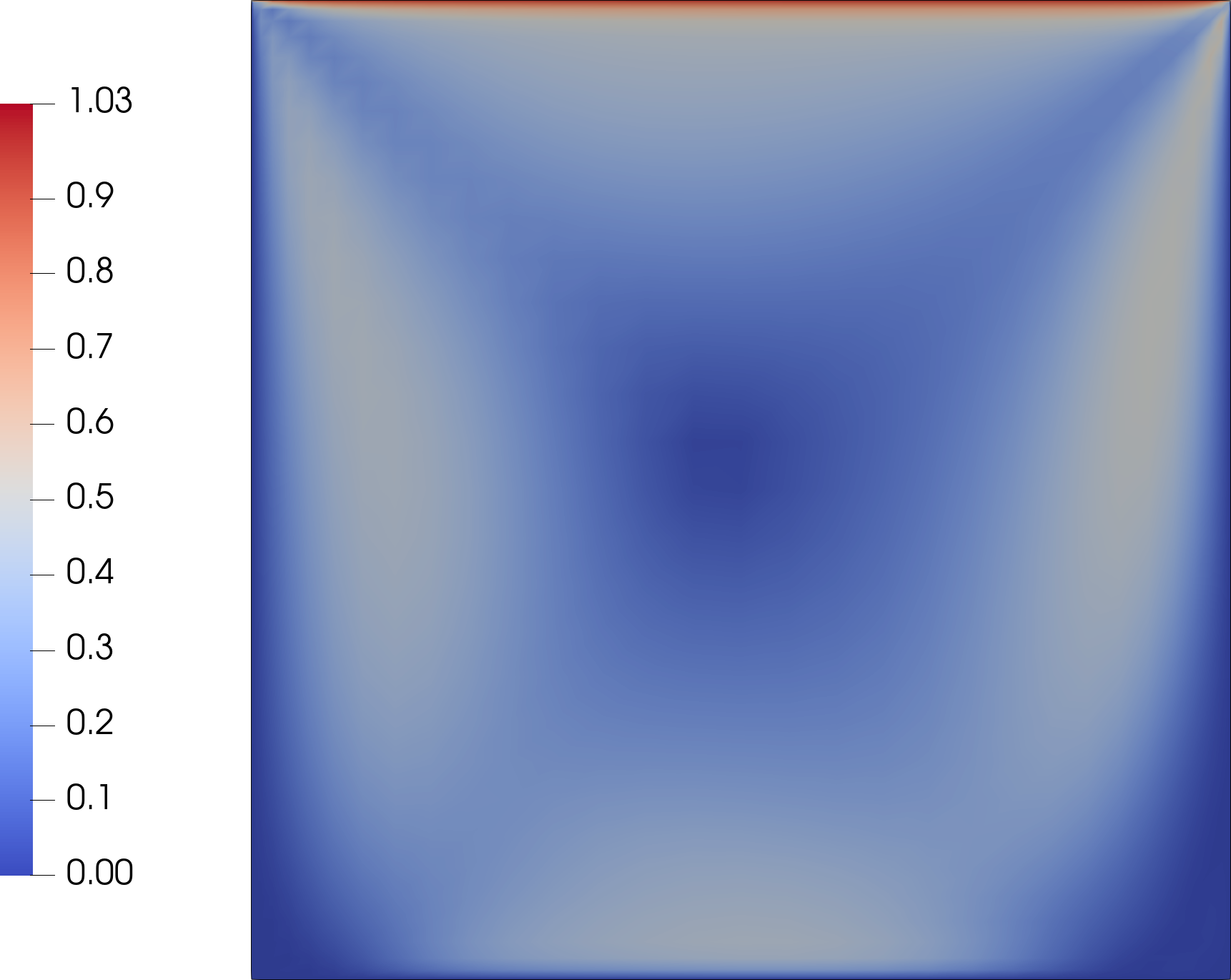}
    \hspace{1cm}
    \includegraphics[height=0.35\textwidth, valign=t]{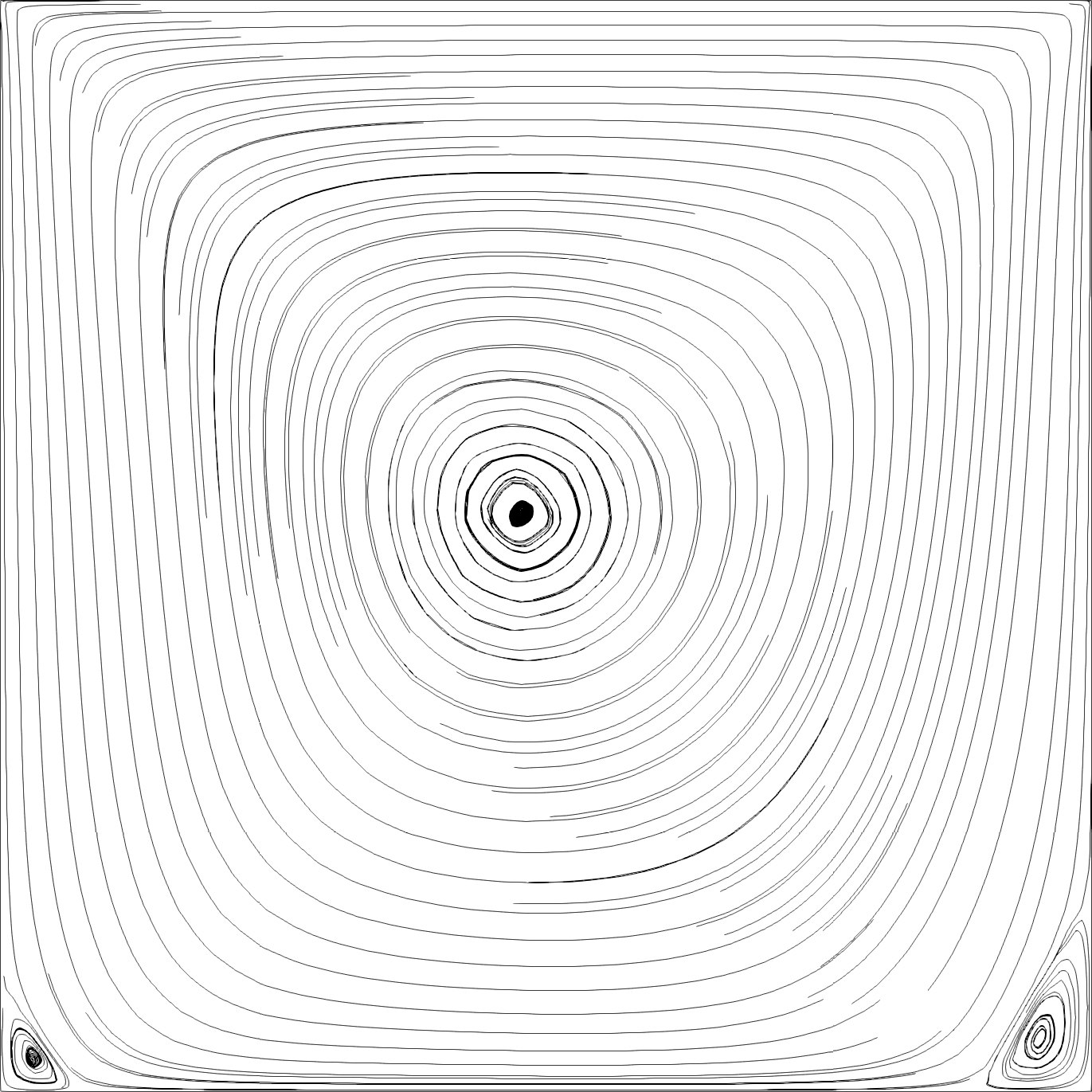}
  }
  \medskip\\
  \subcaptionbox{$B=0.1$ and $\mu=10^{-6}$}[1\textwidth]{
  \includegraphics[height=0.35\textwidth, valign=t]{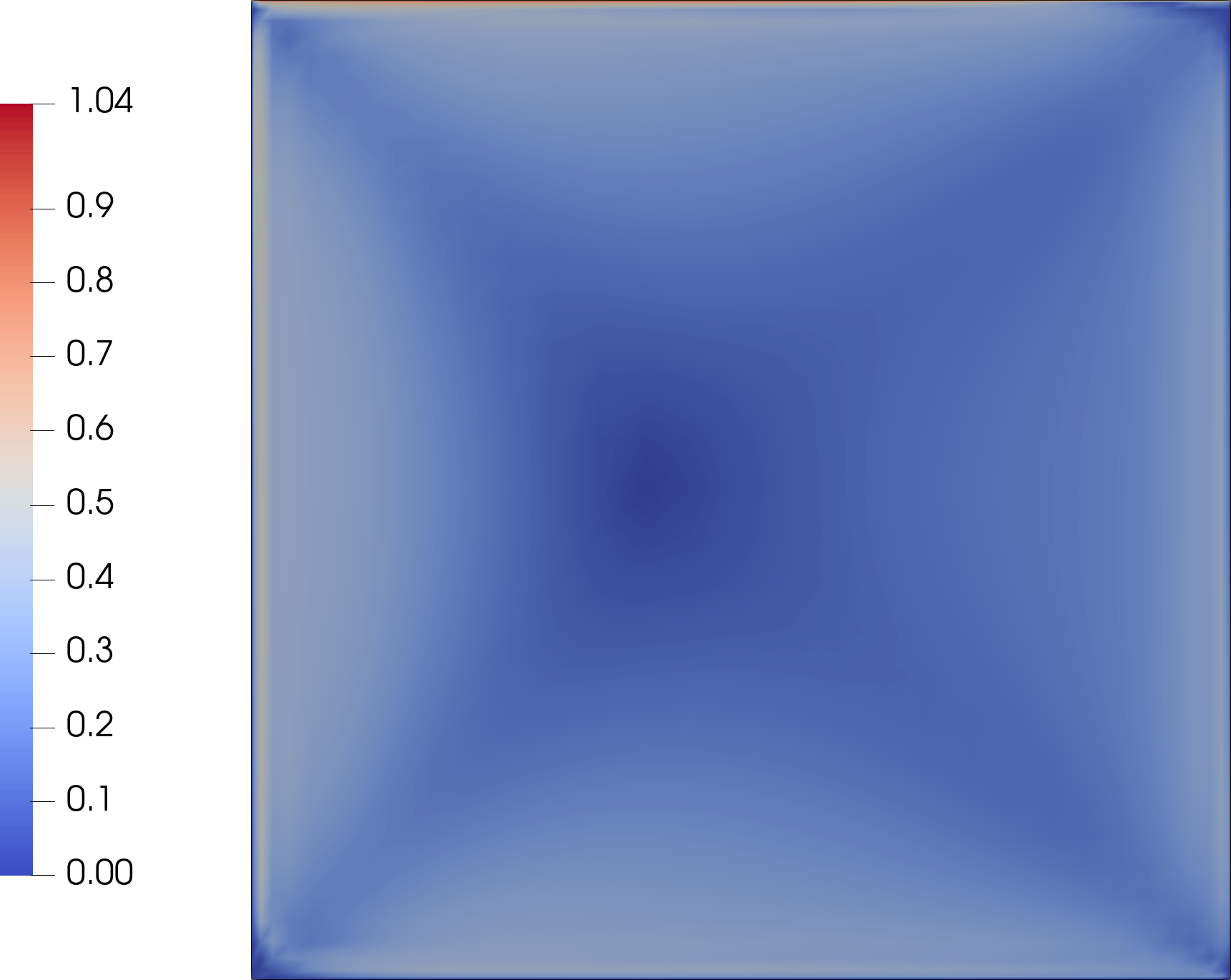}
  \hspace{1cm}
  \includegraphics[height=0.35\textwidth, valign=t]{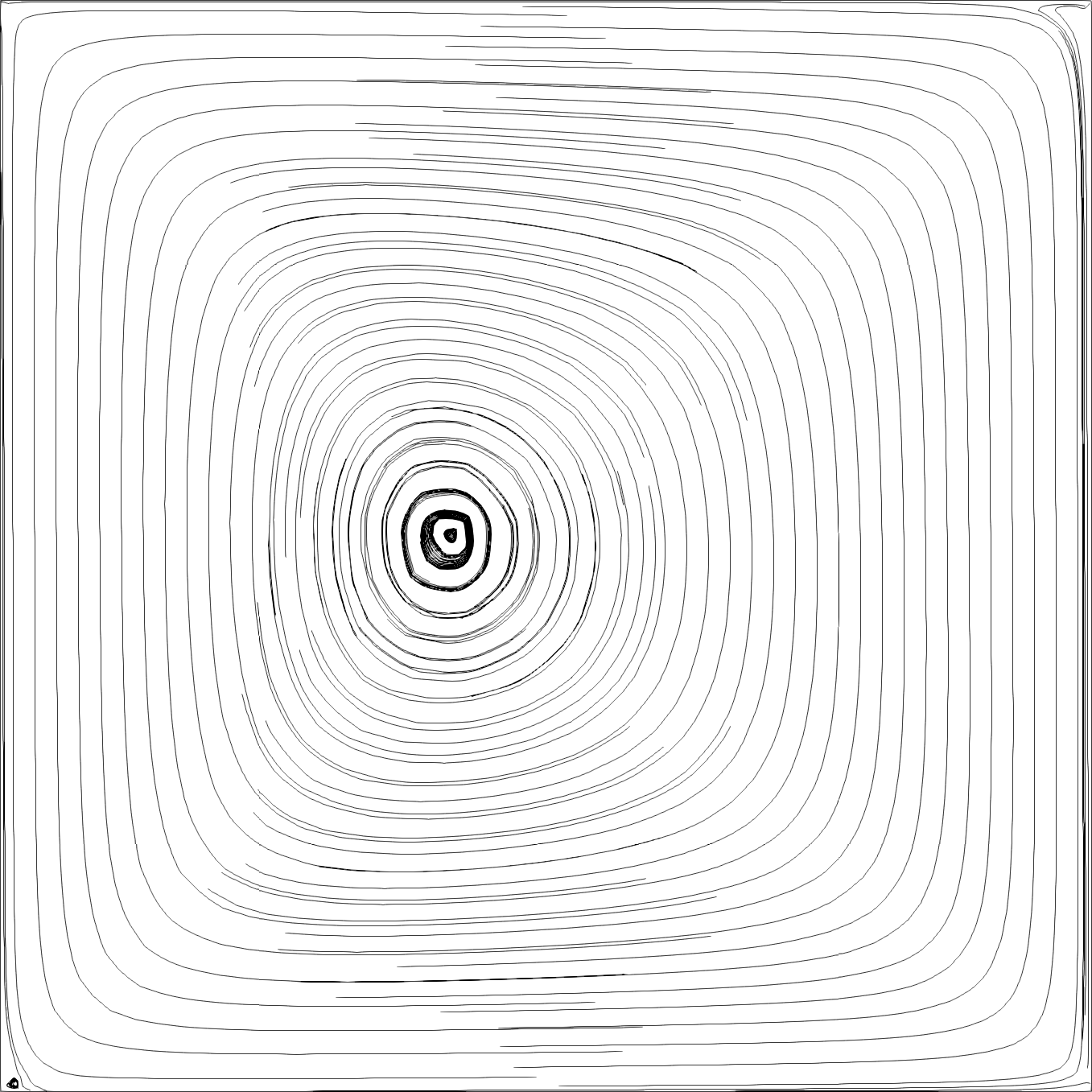}
  }
  \caption{Magnitude of the velocity field (left) and streamlines (right).}
  \label{fig:velocity-B01}
\end{figure}
Figure~\ref{fig:velocity-B01} displays the magnitude of the velocity field and the streamlines for $\mu \in \{1, 10^{-3}, 10^{-6}\}$.
As $\mu$ decreases, the corner vortices gradually shrink while the central vortex expands. We can see how, for $\mu = 10^{-6}$, it has expanded to almost fill the entire domain.
\begin{figure}
  \centering
  \subcaptionbox{$B=0.1$ and $\mu=1$}[1\textwidth]{
    \includegraphics[height=0.35\textwidth, valign=t]{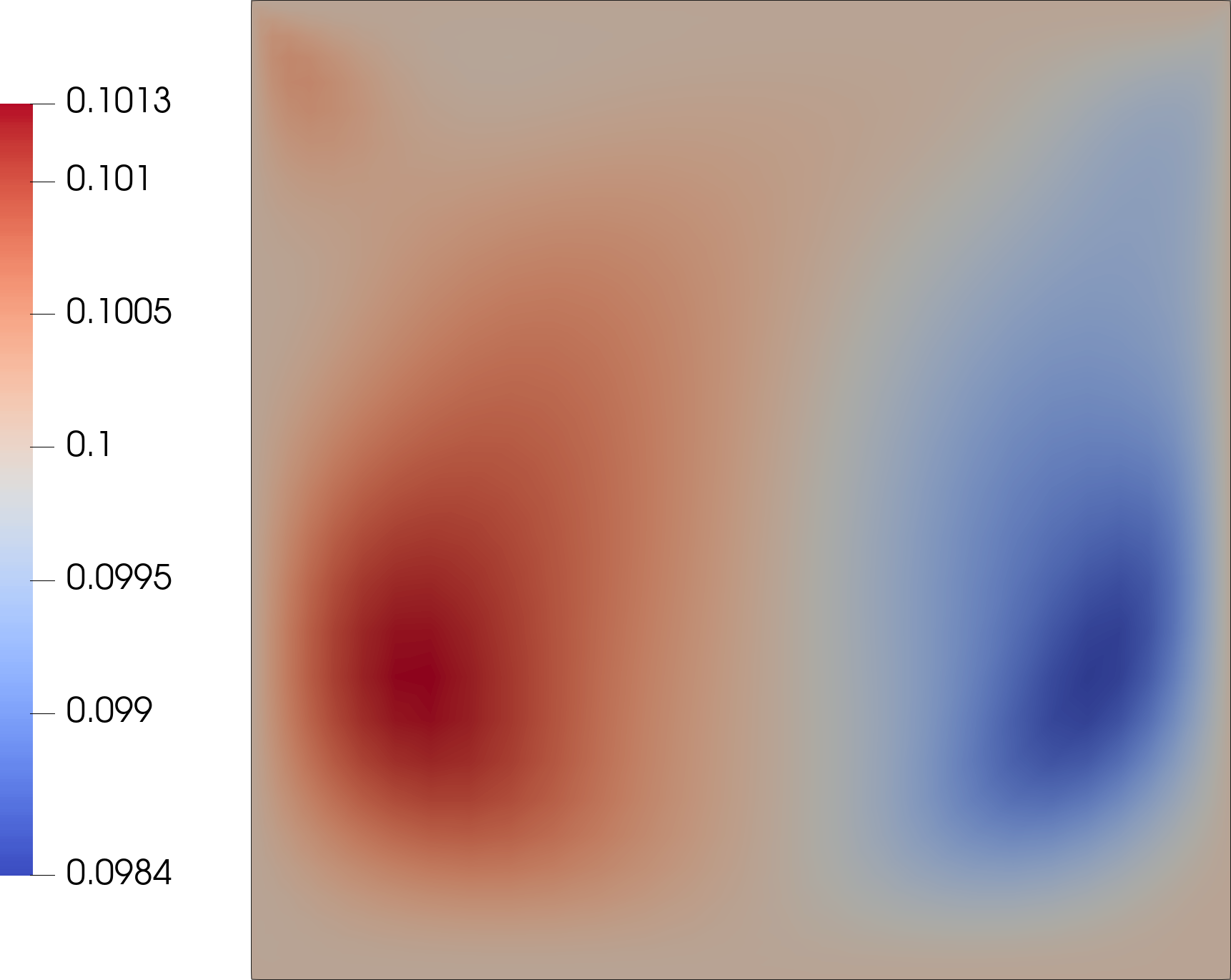}
    \hspace{1cm}
    \includegraphics[height=0.35\textwidth, valign=t]{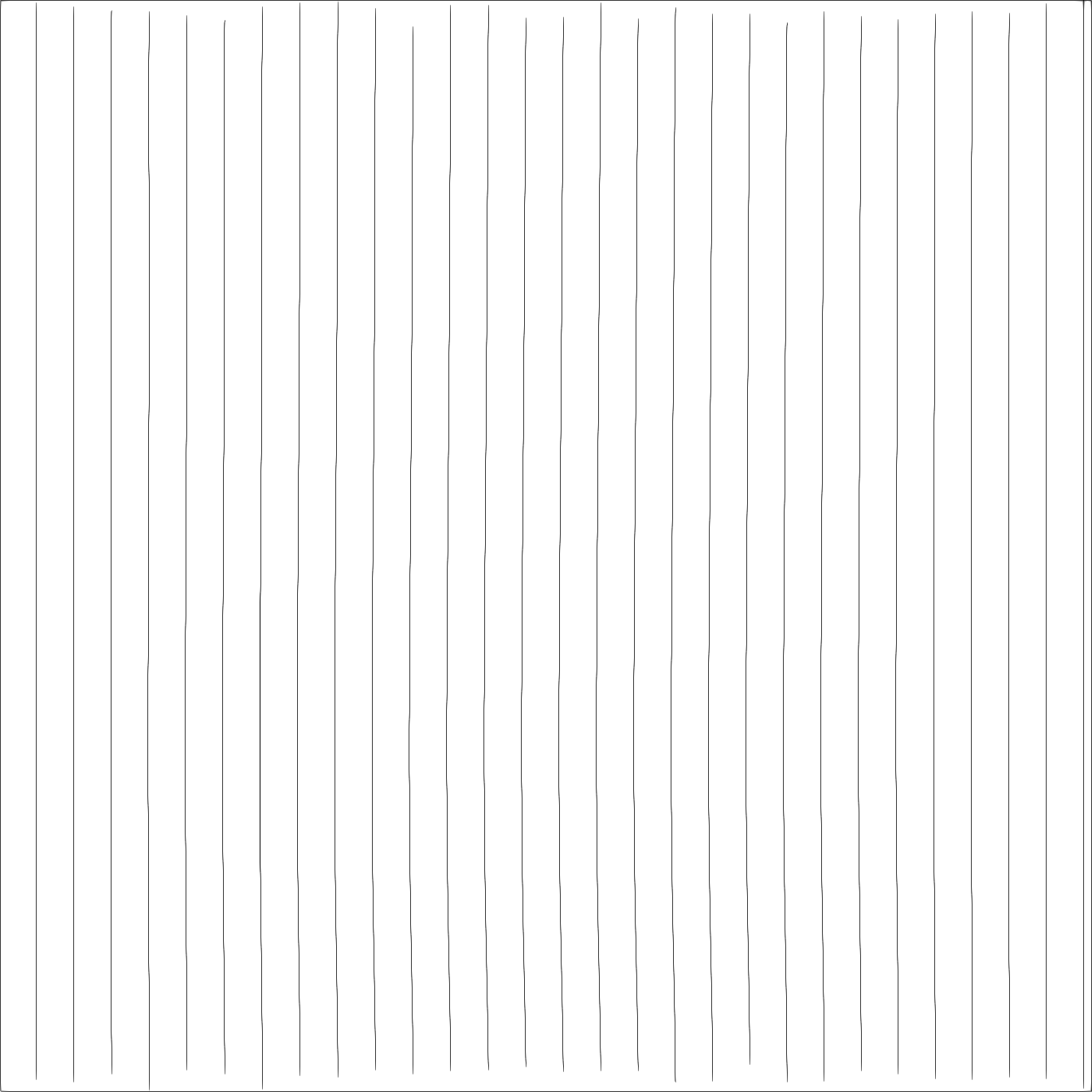}
  }
  \medskip\\
  \subcaptionbox{$B=0.1$ and $\mu=10^{-3}$}[1\textwidth]{
  \includegraphics[height=0.35\textwidth, valign=t]{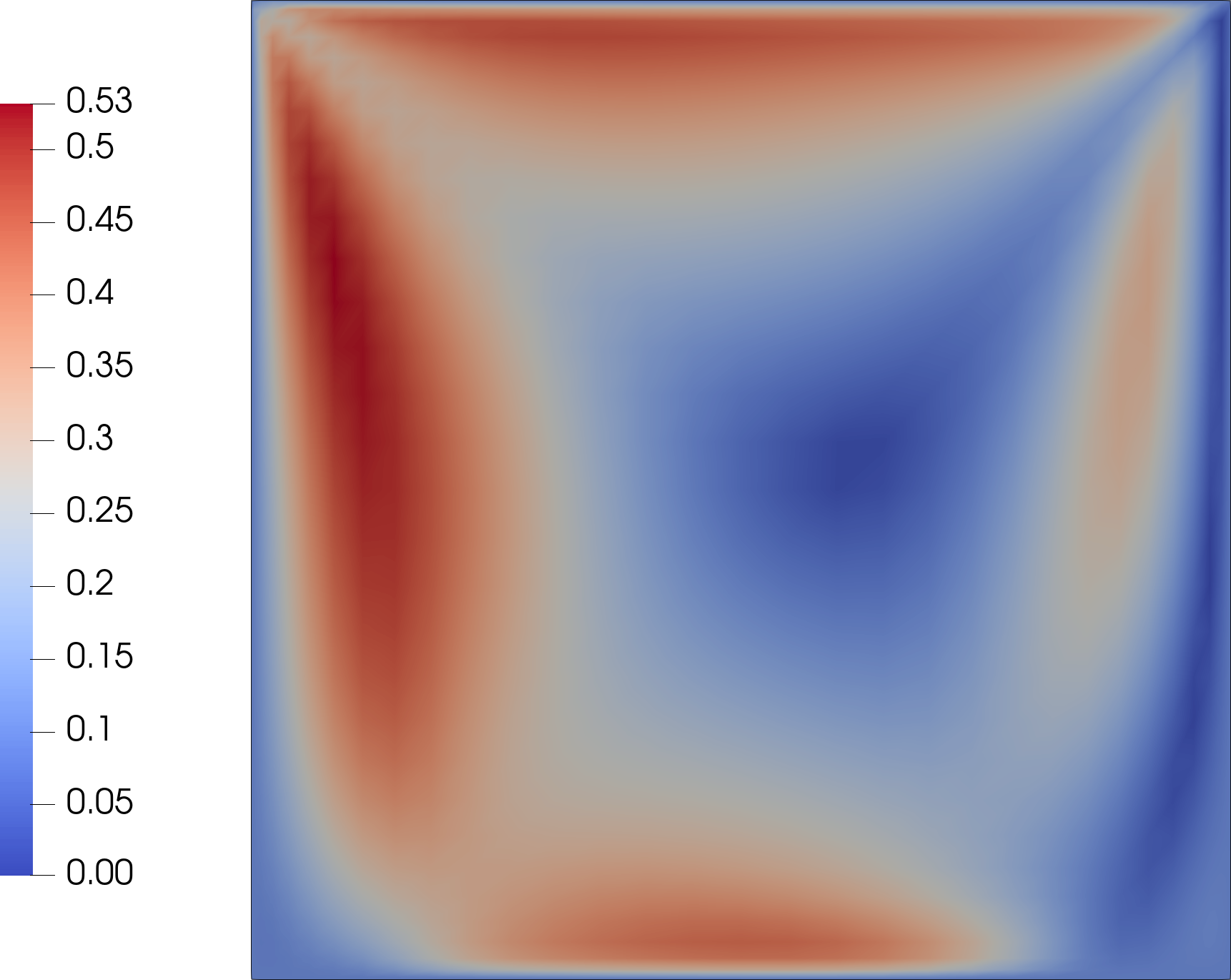}
  \hspace{1cm}
  \includegraphics[height=0.35\textwidth, valign=t]{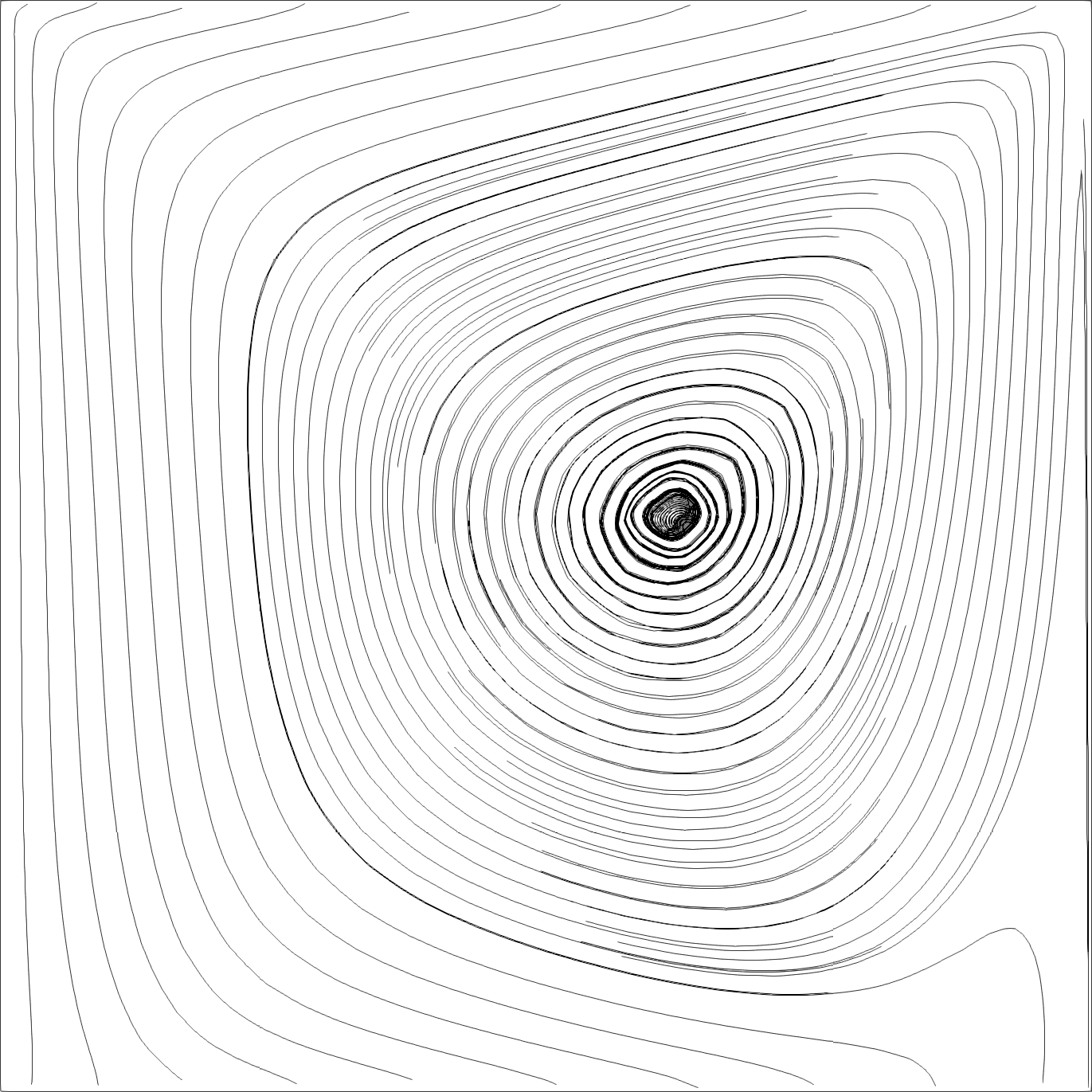}
  }
  \medskip\\
  \subcaptionbox{$B=0.1$ and $\mu=10^{-6}$}[1\textwidth]{
  \includegraphics[height=0.35\textwidth, valign=t]{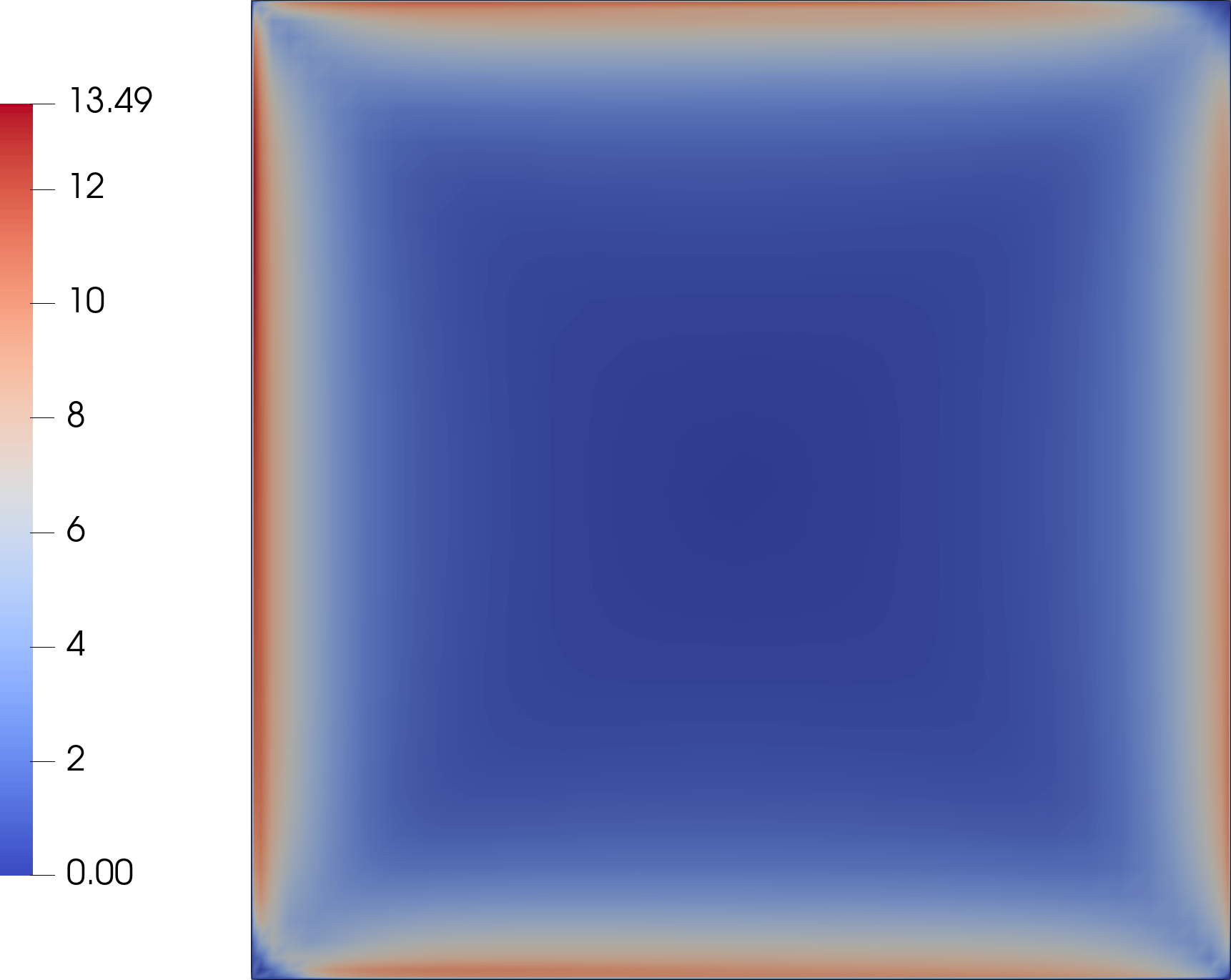}
  \hspace{1cm}
  \includegraphics[height=0.35\textwidth, valign=t]{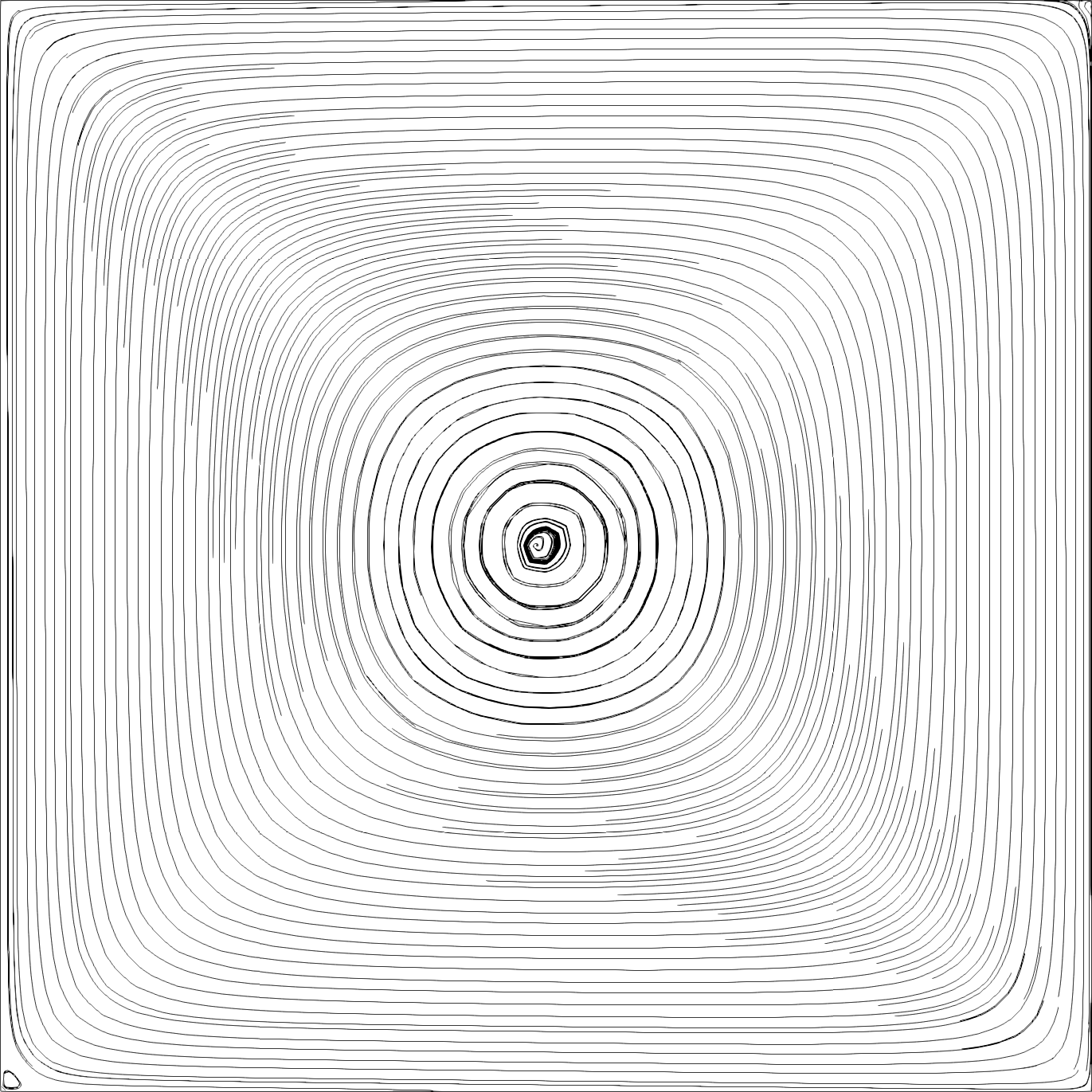}
  }
  \caption{Magnitude of the magnetic field (left) and field lines (right).}
  \label{fig:magnetic-B01}
\end{figure}
Figure~\ref{fig:magnetic-B01} depicts the magnitude and field lines of the corresponding magnetic field.
For $\mu=1$, the magnetic field already presents some tiny gradients, but it remains mostly vertical and close to the initial condition.
For decreasing values of $\mu$, magnetic diffusion becomes weaker and the magnetic field is increasingly transported by the flow. As a result, the magnetic field lines become progressively deformed and may eventually reverse their local direction, forming closed magnetic field lines. This mechanism is responsible for the appearance of the magnetic vortex observed in the simulations.
In particular, for $\mu=10^{-6}$, the magnetic field is strongly advected, leading to a recirculating behavior induced by the velocity field. 
Consequently, since most of the magnetic field lines are closed, the remaining field lines carry the magnetic flux imposed by the boundary conditions, leading to very steep gradients near the boundaries. Moreover, a small recirculation vortex can be observed in the left corner.

\section{Analysis of the scheme}\label{sec:analysis}

This section contains the proof of the results stated in Section ~\ref{sec:discrete.problem:main.results}.

\subsection{Well-posedness}\label{sec:analysis:well-posedness}

In order to prove Theorem~\ref{thm:well-posedness}, we need the following technical result.

\begin{lemma}[$L^2$-boundedness of the velocity interpolator]
  For all $v\in H^1(\Omega)^d$, we have
  \begin{equation}\label{eq:H1-boundedness}
    \text{%
      $\norm{0,h}{\Iu v} \lesssim \norm{H^1(\Omega)^d}{v}$
      and
      $\norm{1,h}{\Iu v} \lesssim \seminorm{H^1(\Omega)^d}{v}$.
    }
  \end{equation}
\end{lemma}

\begin{proof}
  For all $T \in \Th$, we have
  \begin{equation}\label{eq:IU_T_L2}
    \norm{0,T}{\underline{I}^k_{U,T} v}^2
    \overset{\eqref{eq:Iu},\eqref{eq:norm.0.h},\eqref{eq:prod.0.h}}=
    \norm{L^2(T)^d}{\Irt v}^2
    + h_T \sum_{F \in \FT} \norm{L^2(F)^d}{\lproj{k}{F} v - \Irt v}^2
  \end{equation}
  and
  \begin{equation}\label{eq:IU_T_H1}
    \norm{1,T}{\underline{I}^k_{U,T} v}^2
    \overset{\eqref{eq:Iu},\eqref{eq:norm.1.h}}=
    \norm{L^2(T)^{d\times d}}{\grad\Irt v}^2
    + h^{-1}_T \sum_{F \in \FT} \norm{L^2(F)^d}{\lproj{k}{F} v - \Irt v}^2.
  \end{equation}
  To estimate the boundary terms, we start by noticing that
  \begin{equation*} 
    \begin{aligned}
      &\sum_{F \in \FT} \norm{L^2(F)^d}{\lproj{k}{F} v - \Irt v}^2
      \\
      &\quad
      =  \sum_{F \in \FT} \norm{L^2(F)^d}{\lproj{k}{F} v - \lproj{k}{T} v  + \lproj{k}{T} v - v + v - \Irt v}^2
      \\
      &\quad
      \le
      3 \sum_{F \in \FT} \left(
      \norm{L^2(F)^d}{\lproj{k}{F} v - \lproj{k}{T} v}^2
      + \norm{L^2(F)^d}{\lproj{k}{T} v  - v}^2
      + \norm{L^2(F)^d}{ v  - \Irt v}^2
      \right),
    \end{aligned}
  \end{equation*}
  where, in the last passage, we have used the inequality $(a+b+c)^2\le 3 (a^2 + b^2 + c^2)$ valid for all $a,b,c\in\Real$.
  For the first term inside the summation, we can write
  $\norm{L^2(F)^d}{\lproj{k}{F} v - \lproj{k}{T} v}
  = \norm{L^2(F)^d}{\lproj{k}{F} (v - \lproj{k}{T} v)}
  \le \norm{L^2(F)^d}{v - \lproj{k}{T} v}$,
  where we have used, respectively, the polynomial reproduction and boundedness of $\lproj{k}{F}$ in the first and second steps.
  Hence,
  \begin{equation}\label{eq:H1-boundedness:boundary:intermediate}
    \sum_{F \in \FT} \norm{L^2(F)^d}{\lproj{k}{F} v - \Irt v}^2
    \lesssim
      \sum_{F \in \FT}\left(
      \norm{L^2(F)^d}{v - \lproj{k}{T} v}^2
      + \norm{L^2(F)^d}{v - \IRTNT{k+1} v}^2
      \right).
  \end{equation}
  Recalling the approximation properties of $\lproj{k}{T}$ (see, e.g., \cite[Theorem~1.45]{Di-Pietro.Droniou:20}) for the first term in the right-hand side
  and using~\eqref{eq:RTN_bound_F} with $(p,q) = (2,0)$ for the second one, we conclude that
  \begin{equation}\label{eq:H1-boundedness:boundary}
    \sum_{F \in \FT} \norm{L^2(F)^d}{\lproj{k}{F} v - \Irt v}^2
    \lesssim h_T \seminorm{H^1(T)^d}{v}^2.
  \end{equation}

  We next notice that
  \[
  \norm{L^2(T)^d}{\Irt v}^2
  \lesssim
  \norm{L^2(T)^d}{v}^2
  + \norm{L^2(T)^d}{\Irt v - v}^2
  \overset{\eqref{eq:RTN_bound}}{\lesssim}
  \norm{L^2(T)^d}{v}^2
  + h^{2}_T \seminorm{H^1(T)^d}{v}^2
  \]
  and
  \[
  \norm{L^2(T)^{d\times d}}{\grad \Irt v}^2
  \lesssim
  \seminorm{H^1(T)^d}{v}^2
  + \seminorm{H^1(T)^d}{\Irt v - v}^2
  \overset{\eqref{eq:RTN_bound}}{\lesssim} \seminorm{H^1(T)^d}{v}^2.
  \]
  Using~\eqref{eq:H1-boundedness:boundary} together with, respectively, the first and second bound above in \eqref{eq:IU_T_L2} and \eqref{eq:IU_T_H1}, we obtain that $\norm{0,T}{\underline{I}^k_{U,T} v}^2 \lesssim \norm{L^2(T)^d}{v}^2 + h_T^2 \seminorm{H^1(T)^d}{v}^2$ and $\norm{1,T}{\underline{I}^k_{U,T} v}^2 \lesssim  \seminorm{H^1(T)^d}{v}^2$.
  Summing over $T\in \Th$ and recalling that $h_T \leq \operatorname{diam}(\Omega) \lesssim 1$ yields the desired result.
\end{proof}

\begin{proof}[Proof of Theorem~\ref{thm:well-posedness}]
  Problem~\eqref{eq:discrete.bis} corresponds to a system of nonlinear ODEs in a finite-dimensional space. Its right-hand side is continuous in time and polynomial, hence locally Lipschitz, in the unknowns. The Cauchy--Lipschitz theorem therefore yields a unique local-in-time $C^1$ solution $(\uh,\bh)$.
    Since the phase space is finite-dimensional, the a priori estimate~\eqref{eq:energy_bound} prevents finite-time blow-up of the solution norm; the standard continuation criterion for ODEs therefore extends the solution up to $\tF$. Continuity of the recovered pressure and magnetic pressure then follows from the inf-sup condition~\eqref{eq:inf-sup}.
  Let us then prove \eqref{eq:energy_bound}.
  Taking $\vh = \uh$ in \eqref{eq:discrete.bis:momentum} and $\wh = \bh$ in \eqref{eq:discrete.bis:maxwell},
  noticing that $t_h(\uh,\uh,\uh) = t_h(\uh,\bh,\bh) = 0$ by the non-dissipativity property \eqref{eq:non-dissipativity} of $\th$,
  and using the coercivity \eqref{eq:ah:norm.equivalence} of $a_h$, we infer
  \[
  \frac12\partial_{t} \norm{0,h}{\uh}^2
  + \nu \norm{1,h}{\uh}^2
  + \seminorm{\beta,h}{\uh}^2
  - \th(\bh,\bh, \uh)
  \lesssim \int_\Omega f \cdot u_h
  \]
  and
  \[
  \frac12 \partial_{t} \norm{0,h}{\bh}^2
  + \mu \norm{1,h}{\bh}^2
  + \seminorm{\gamma,h}{\bh}^2
  - \th(\bh,\uh, \bh)
  \lesssim \int_\Omega g \cdot b_h.
  \]
  Summing the above relations and using the skew-symmetry \eqref{eq:skew-simmetry} of the trilinear form $\th$ to write $-\th(\bh,\uh,\bh) = \th(\bh,\bh,\uh) $, we get
  \[
  \frac12\partial^{}_{t} \left(
  \norm{0,h}{\uh}^2 + \norm{0,h}{\bh}^2
  \right)
  + \nu \norm{1,h}{\uh}^2
  + \mu \norm{1,h}{\bh}^2
  + \seminorm{\beta,h}{\uh}^2
  + \seminorm{\gamma,h}{\bh}^2
  \lesssim \int_\Omega f \cdot u_h
  + \int_\Omega g \cdot b_h.
  \]
  Using Young's inequalities, as well as the fact that $\norm{L^2(\Omega)^d}{v_h} \le \norm{0,h}{\vh}$ for all $\vh \in \U$ (by \eqref{eq:norm.0.h} and \eqref{eq:prod.0.h}), we can go on writing
  \begin{multline*}
    \frac12\partial^{}_{t} \left(
    \norm{0,h}{\uh}^2 + \norm{0,h}{\bh}^2
    \right)
    + \nu \norm{1,h}{\uh}^2
    + \mu \norm{1,h}{\bh}^2
    + \seminorm{\beta,h}{\uh}^2
    + \seminorm{\gamma,h}{\bh}^2
    \\
    \lesssim
    \frac12\norm{L^2(\Omega)^d}{f}^2
    + \frac12\norm{L^2(\Omega)^d}{g}^2
    + \frac12\norm{0,h}{\uh}^2
    + \frac12\norm{0,h}{\bh}^2.
  \end{multline*}
  Integrating the above inequality over $[0,t]$ for an arbitrary $t \in [0,\tF]$, we get
  \[
  \begin{aligned}
    &\norm{0,h}{\uh(t)}^2
    + \norm{0,h}{\bh(t)}^2
    + \int^t_0 \left(
    \nu \norm{1,h}{\uh(\tau)}^2
    + \mu \norm{1,h}{\bh(\tau)}^2
    + \seminorm{\beta,h}{\uh(\tau)}^2
    + \seminorm{\gamma,h}{\bh(\tau)}^2
    \right) \, d\tau
    \\
    &\qquad
    \lesssim
    \frac12\norm{L^2(0,t;L^2(\Omega)^d)}{f}^2
    + \frac12\norm{L^2(0,t;L^2(\Omega)^d)}{g}^2
    + \frac12\int^t_0 \left(
    \|\uh(\tau) \|^2_{0,h}
    + \| \bh(\tau) \|^2_{0,h}
    \right) \, d\tau
    \\
    &\qquad\quad
    + \norm{0,h}{\Iu u_0}^2
    + \norm{0,h}{\Iu b_0}^2.
  \end{aligned}
  \]
  Invoking Gronwall's inequality~\cite[Proposition~2.1]{Emmrich:99} after noticing that, by the continuity~\eqref{eq:H1-boundedness} of the interpolator and the Poincar\'e inequality in $H^1_0(\Omega)$, $\norm{0,h}{\Iu u_0} \lesssim \norm{H^1(\Omega)^d}{u_0} \lesssim \seminorm{H^1(\Omega)^d}{u_0}$ and, similarly, $\norm{0,h}{\Iu b_0} \lesssim \seminorm{H^1(\Omega)^d}{b_0}$, and recalling the definition~\eqref{eq:energy.norm} of the energy norm, the estimate~\eqref{eq:energy_bound} follows.
\end{proof}

\subsection{Basic error estimates}\label{sec:analysis:error.estimate}

\begin{proof}[Proof of Theorem~\ref{thm:error.estimate}]
  Set, for the sake of brevity,  $\hu{u}_h \coloneqq \Iu u$ and
  $\hu{b}_h \coloneqq \Iu b$.
  Since the discrete formulations~\eqref{eq:discrete} and~\eqref{eq:discrete.bis} are equivalent, we can proceed with the analysis of the latter.
  We start by summing the equations \eqref{eq:discrete.bis:momentum} and \eqref{eq:discrete.bis:maxwell}  to infer, for all $\vh,\wh \in \Zh$,
  \[
  \begin{aligned}
    &(\partial^{}_{t} \uh,\vh)_{0,h}
    + (\partial^{}_{t} \bh, \wh)_{0,h}
    + \nu a_h(\uh,\vh)
    + \mu a_{h}(\bh, \wh)
    \\
    &\quad
    + \th(\uh,\uh,\vh)
    + j_{ {\beta},h}(\uh,\vh)
    - \th(\bh,\bh, \vh)
    \\
    &\quad
    + \th(\uh,\bh,\wh)
    + j_{\gamma,h}(\bh,\wh)
    - \th(\bh,\uh, \wh)
    = \int_\Omega f \cdot v_h
    + \int_\Omega g \cdot w_h.
  \end{aligned}
  \]
  Subtracting from both sides the expression on the left-hand side with $\underline{u}_h$ and $\underline{b}_h$ respectively replaced by $\hu{u}_h$ and $\hu{b}_h$, using the fact that~\eqref{eq:modified:momentum} and~\eqref{eq:modified:maxwell} hold almost everywhere to substitute $f$ and $g$ in the right-hand side, invoking \eqref{eq:vh.zero.div} to justify the removal of pressure-related terms,
  and recalling the expression~\eqref{eq:Eh} of the global consistency error, we obtain the following equation for $(\eu, \eb)$ (cf.~\eqref{eq:approximation.errors}):
  \begin{equation}\label{eq:basic_error}
    \begin{aligned}
      &(\partial_{t} \eu,\vh)_{0,h}
      + (\partial_{t} \eb, \wh)_{0,h}
      + \nu a_h(\eu,\vh)
      + \mu a_{h}(\eb, \wh)
      \\
      &\quad
      + \th(\uh,\uh,\vh) - \th(\hu{u}_h, \hu{u}_h ,\vh)
      +\th(\uh,\bh,\wh) - \th(\hu{u}_h, \hu{b}_h ,\wh)
      \\
      &\quad
      + j_{ {\beta},h}(\eu,\vh) + \th(\hu{b}_h ,\hu{b}_h, \vh)
      - \th(\underline{ {b}}_h,\bh, \vh)
      \\
      &\quad
      + j_{\gamma,h}(\eb,\wh) + \th(\hu{b}_h, \hu{u}_h, \wh)
      - \th(\bh,\uh,\wh)
      =\mathcal{E}_h(\vh,\wh).
    \end{aligned}
  \end{equation}
  Take now $(\vh, \wh) = (\eu, \eb)$ in \eqref{eq:basic_error}, which is a valid choice by \eqref{eq:approximation.errors}.
  Recalling definition~\eqref{eq:j-seminorm} of the seminorms $\seminorm{\beta,h}{\cdot}$ and $\seminorm{\gamma,h}{\cdot}$,
  and rearranging, we obtain
  \[
  \begin{aligned}
    &\frac12 \partial_{t} \left( \|\eu\|^2_{0,h} +  \| \eb\|^2_{0,h} \right)
    + \nu a_h(\eu,\eu)
    + \mu a_h(\eb,\eb)
    + | \eu |^2_{ {\beta},h}
    + | \eb |^2_{\gamma,h}
    \\
    &\quad
    =\mathcal{E}_h(\eu,\eb)
    + \underbrace{\th(\hu{u}_h, \hu{u}_h ,\eu) - \th(\uh,\uh,\eu)}_{\term_1}
    + \underbrace{\th(\hu{u}_h, \hu{b}_h,\eb) - \th(\uh,\bh,\eb)}_{\term_2}
    \\
    &\qquad
    + \underbrace{%
      \th(\underline{ {b}}_h,\bh, \eu) - \th(\hu{b}_h ,\hu{b}_h, \eu)
      + \th(\underline{ {b}}_h,\uh,\eb) - \th(\hu{b}_h, \hu{u}_h, \eb).
    }_{\term_3}
  \end{aligned}
  \]
  Adding and subtracting $\th(\uh, \hu{u}_h,\eu)$ to $\term_1$,
  $\th(\uh, \hu{b}_h, \eb)$ to $\term_2$,
  and using the non-dissipativity property~\eqref{eq:non-dissipativity} of $t_h$, we get
  \[
  \begin{aligned}
    \term_1 &=
    - \th(\eu, \hu{u}_h ,\eu)
    - \cancel{\th(\uh,\eu,\eu)}
    = - \th(\eu, \hu{u}_h ,\eu),
    \\
    \term_2 &=
    - \th(\eu, \hu{b}_h,\eb)
    - \cancel{\th(\uh,\eb,\eb)}
    = - \th(\eu, \hu{b}_h,\eb).
  \end{aligned}
  \]
  We next  add and subtract $ \th(\bh,\hu{b}_h,\eu) + \th(\bh,\hu{u}_h,\eb)$ to $\term_3$ to obtain
  \[
  \term_3
  = \cancel{\th(\bh,\eb,\eu)} + \th(\eb, \hu{b}_h,\eu)
  +\cancel{\th(\bh,\eu, \eb)} + \th(\eb,\hu{u}_h, \eb),
  \]
  where we have used the skew-symmetry property \eqref{eq:skew-simmetry} in the cancellation.
  Gathering the above relations, we arrive at
  \[
  \begin{aligned}
    &\frac12 \partial_{t} \left(
    \norm{0,h}{\eu}^2 + \norm{0,h}{\eb}^2
    \right)
    + \nu a_h(\eu,\eu)
    + \mu a_h(\eb,\eb)
    + \seminorm{\beta,h}{\eu}^2
    + \seminorm{\gamma,h}{\eb}^2
    \\
    &\quad
    = \mathcal{E}_h(\eu,\eb)
    - \th(\eu, \hu{u}_h ,\eu)
    - \th(\eu, \hu{b}_h ,\eb)
    + \th(\eb, \hu{b}_h ,\eu)
    + \th(\eb,\hu{u}_h , \eb).
  \end{aligned}
  \]
  Using the coercivity~\eqref{eq:ah:norm.equivalence} of $a_h$ together with the boundedness \eqref{thbound.0h} of the trilinear form $\th$ and~\eqref{eq:W.1.infty.boundedness.IUh} of the interpolator $\Iu$, we go on writing
  \[
  \begin{aligned}
    &\frac12 \partial_{t} \left(
    \norm{0,h}{\eu}^2
    + \norm{0,h}{\eb}^2
    \right)
    + \nu \norm{1,h}{\eu}^2
    + \mu \norm{1,h}{\eb}^2
    + \seminorm{\beta,h}{\eu}^2
    + \seminorm{\gamma,h}{\eb}^2
    \\
    &\quad
    \lesssim \left|\mathcal{E}_h(\eu,\eb)\right|
    + \seminorm{W^{1,\infty}(\Omega)^d}{u} \norm{0,h}{\eu}^2
    + \seminorm{W^{1,\infty}(\Omega)^d}{b} \norm{0,h}{\eu} \norm{0,h}{\eb}
    + \seminorm{W^{1,\infty}(\Omega)^d}{u} \norm{0,h}{\eb}^2.
  \end{aligned}
  \]
  Applying Young's inequality $ab \le \frac12 a^2 + \frac{1}{2} b^2$ to the third term in the right-hand side and integrating in time from $0$ to  $t \in [0,\tF]$, we get
  \[
  \begin{aligned}
    &\norm{0,h}{\eu(t)}^2
    + \norm{0,h}{\eb(t)}^2
    \\
    &\quad
    + \int^{t}_{0} \left(
    \nu \norm{1,h}{\eu(\tau)}^2
    + \mu \norm{1,h}{\eb(\tau)}^2
    + \seminorm{\beta,h}{\eu(\tau)}^2
    + \seminorm{\gamma,h}{\eb(\tau)}^2
    \right) d\tau
    \\
    &\quad
    \lesssim
    \int^{t}_{0} \left| \mathcal{E}_h(\eu(\tau),\eb(\tau)) \right| d\tau
    \\
    &\qquad
    + \int^t_0 \left(
    \seminorm{W^{1,\infty}(\Omega)^d}{u(\tau)}
    + \seminorm{W^{1,\infty}(\Omega)^d}{b(\tau)}
    \right)
    \left(
    \norm{0,h}{\eu(\tau)}^2
    + \norm{0,h}{\eb(\tau)}^2
    \right) d\tau.
  \end{aligned}
  \]
  The conclusion now follows using Gronwall's inequality as in Theorem~\ref{thm:well-posedness}.
\end{proof}

\subsection{Convergence rates}\label{sec:analysis:convergence.rates}

The following bounds for the time consistency error~\eqref{eq:E.time.h} and the diffusive error~\eqref{eq:E.diff.h} can be found in \cite[Lemma~4.8]{Beirao-da-Veiga.Di-Pietro.ea:25} and \cite[Theorem~14]{Botti.Botti.ea:26}, respectively:
\begin{itemize}
\item Let $w\in H^{1}(0,\tF; H^{1}_{0}(\Omega)^d\cap H^{k+1}(\Th)^d)$.
  Then, for almost every $t \in (0,\tF)$, we have
  \begin{equation}\label{eq:estimate.Err.time}
    |\Err{\rm time}(w,\vh)|
    \lesssim
    \left(
    \sum_{T \in \Th} h_T^{2(k+1)}
    \Seminorm{H^{k+1}(T)^d}{\frac{d w}{dt}}^2
    \right)^{\frac12} \norm{0,h}{\vh}\qquad\forall\vh\in \Ub.
  \end{equation}
\item For all $w \in H^1_0(\Omega)^d\cap H^{k+2}(\Th)^d$, we have
  \begin{equation}\label{eq:estimate.Err.diff.h}
    |\Err{\rm diff}(w,\vh)|
    \lesssim
    \left(
    \sum_{T \in \Th} h_T^{2(k+1)} \seminorm{H^{k+2}(T)^d}{w}^2
    \right)^{\frac12}
    \norm{1,h}{\vh}\qquad\forall\vh\in \Ub.
  \end{equation}
\end{itemize}

The following lemma contains a regime-dependent estimate for the convective consistency error~\eqref{eq:E.conv.h}.
This estimate accounts for orders ranging from $h^{k+\frac12}$ (dominant convection) to $h^{k+1}$ (dominant diffusion), depending on the value of a local P\'eclet number.
When estimating the convective contributions to the global consistency error~\eqref{eq:Eh}, this P\'eclet number will embody either a Reynolds- or a Hartmann-type measure.

\begin{lemma}[Regime-dependent estimates of the convective errors]
  Denote by $\alpha \coloneqq (\alpha_T)_{T \in \Th}$ a family of strictly positive real numbers and let $\eta > 0$.
  Then, for all $v,w \in H_0^1(\Omega)^d \cap W^{1,\infty}(\Omega)^d \cap H^{k+1}(\Th)^d$ such that $\divergence v = \divergence w =0$ and all $\underline{z}_h \in \Zh$, the following estimates hold with a hidden constant additionally independent of $\eta$, $\alpha$, $v$, and $w$. For the sake of brevity, we write $\Pe_T$ instead of $\Pe_T(\eta,w,\alpha_T)$ and $\chi$ instead of $\chi(\eta,w,\alpha)$:
  \begin{equation}\label{eq:E.conv:estimate}
    \begin{aligned}
      &| \Err{\rm conv}(w,v,\underline{z}_h) |
      \\
      &\quad
      \lesssim
      \left[
        \sum_{T\in \Th} h^{2(k+1)}_T \left(
        \seminorm{W^{1,\infty}(T)^d}{v} \seminorm{H^{k+1}(T)^d}{w}
        + \seminorm{W^{1,\infty}(T)^d}{w} \seminorm{H^{k+1}(T)^d}{v}
        \right)^2
        \right]^{\frac12} \|  {z}_h \|_{ L^2(\Omega)^d}
      \\
      &\qquad
      + (1+\chi)^{\frac12}
      \left[
        \sum_{T\in \Th}  h^{2k+1}_T
        \min(1,\Pe_T)
        \norm{L^{\infty}(T)^d}{w} \seminorm{H^{k+1}(T)^d}{v}^2
        \right]^{\frac12}
      \left(
      \seminorm{\alpha,h}{\underline{z}_h}^2
      + \eta \norm{1,h}{\underline{z}_h}^2
      \right)^{\frac12}
    \end{aligned}
  \end{equation}
  and
  \begin{equation}\label{eq:j.alpha:estimate}
    |j_{ {\alpha},h}(\IUh v,\underline{z}_h)|
    \lesssim \left[
      \sum_{T \in \Th} h_T^{2k+1} \min(1,\Pe_T) \alpha_T  \seminorm{H^{k+1}(T)^d}{v}^2
      \right]^{\frac12}\,
    \left(
    \seminorm{\alpha,h}{\underline{z}_h}^2
    + \eta \norm{1,h}{\underline{z}_h}^2
    \right)^{\frac12}.
  \end{equation}
\end{lemma}

\begin{proof}
  Let, for the sake of brevity $\hu{v}_h \coloneqq \IUh v$ and $\hu{w}_h \coloneqq \IUh w$. Expanding $t_h$ according to its definition \eqref{eq:th} in the convective consistency error \eqref{eq:E.conv.h} and rearranging the boundary terms, we obtain
  \[
  \begin{aligned}
    &\Err{\rm conv}(w,v,\underline{z}_h)
    \\
    &\quad=
    \underbrace{%
      \sum_{T\in \Th} \left[
        \int_T (w \cdot \grad) v \cdot  z_T
        - \int_T (\widehat{w}_T \cdot \grad) \widehat{v}_T \cdot  z_T
        - \sum_{F\in \mathcal{F}_T} \int_F (\widehat{w}_T \cdot \nTF) ( \widehat{v}_F - \widehat{v}_T) \cdot  z_T
        \right]
    }_{\term_1}
    \\
    &\qquad
    \underbrace{%
      -\frac12 \sum_{T\in \Th} \sum_{F\in \mathcal{F}_T} \int_F (\widehat{w}_T \cdot \nTF) ( \widehat{v}_F - \widehat{v}_T) \cdot ( z_F -  z_T).
    }_{\term_2}
  \end{aligned}
  \]
  We write $\term_i = \sum_{T \in \Th} \term_i(T)$ and estimate each contribution separately.
  Letting $\overline{w}_T \coloneqq \lproj{0}{T} \widehat{w}_T$, we notice that
  \[
  \begin{aligned}
    \int_T (\overline{w}_T \cdot \grad) v \cdot  z_T
    &= - \int_T v \cdot [\divergence ( z_T \otimes \overline{w}_T )]
    + \sum_{F \in \mathcal{F}_T} \int_F (\overline{w}_T \cdot \nTF) (v \cdot  z_T)
    \\
    &= - \int_T \lproj{k-1}{T} v \cdot [\divergence ( z_T \otimes \overline{w}_T )]
    + \sum_{F \in \mathcal{F}_T} \int_F (\overline{w}_T \cdot \nTF) (\lproj{k}{F} v \cdot  z_T)
    \\
    \overset{\eqref{eq:IRTNT},\eqref{eq:Iu}}&= - \int_T \widehat{v}_T \cdot [\divergence ( z_T \otimes \overline{w}_T )]
    + \sum_{F \in \mathcal{F}_T} \int_F (\overline{w}_T \cdot \nTF) (\widehat{v}_F \cdot  z_T),
  \end{aligned}
  \]
  where the first equality follows from an integration by parts while, in the second equality, we have used the fact that $z_T \in \Poly{k}(T)^d$ (by Remark~\ref{rem:div-free-subspace}), which gives $\divergence ( z_T \otimes \overline{w}_T) \in \Poly{k-1}(T)^d$  and $(\overline{w}_T\cdot n_{TF}) z_T|_F\in \Poly{k}(F)^d$ and justifies the insertion of projectors.
  Integrating by parts the first term in the right-hand side of the above expression and rearranging, we obtain
  \[
  \int_T (\overline{w}_T \cdot \grad) (v- \widehat{v}_T) \cdot  z_T
  - \sum_{F \in \FT} \int_F (\overline{w}_T \cdot \nTF) (\widehat{v}_F - \widehat{v}_T ) \cdot  z_T
  = 0.
  \]
  Subtracting this quantity from $\term_1(T)$, then adding and subtracting $\int_T (\widehat{w}_T \cdot \grad)  v \cdot  z_T$, we arrive at
  \[
  \begin{aligned}
    \term_1(T)
    &= \int_T \left[( w - \widehat{w}_T) \cdot \grad\right] v \cdot  z_T
    + \int_T \left[(\widehat{w}_T - \overline{ w}_T) \cdot \grad\right]( v - \widehat{v}_T) \cdot  z_T
    \\
    &\quad
    - \sum_{F \in \FT} \int_F (\widehat{w}_T - \overline{ w}_T) \cdot \nTF\, (\widehat{v}_F - \widehat{v}_T) \cdot  z_T.
  \end{aligned}
  \]
  Proceeding as in \cite[Lemma~4.10]{Beirao-da-Veiga.Di-Pietro.ea:25} via H\"older inequalities and approximation properties of $L^2$-orthogonal projectors to bound the various factors in the above expression, we get
   \begin{equation}\label{eq:th:consistency:T1+T2}
    | \term_1 |
    \lesssim
      \left[\sum_{T \in \Th}
        h_T^{2(k+1)} \left(
        \seminorm{W^{1,\infty}(T)^d}{v}\seminorm{H^{k+1}(T)^d}{w}
        + \seminorm{W^{1,\infty}(T)^d}{w}\seminorm{H^{k+1}(T)^d}{v}
        \right)^2\right]^{\frac12}
    \|  z_h \|_{ L^2(\Omega)^d}.
  \end{equation}

  Let us now consider $\term_2(T)$ for a generic $T \in \Th$.
  Using H\"{o}lder inequalities and noticing that $\widehat{w}_T \cdot n_{TF}=\lproj{k}{F}(w\cdot n_{TF})$ (see \eqref{eq:IRTNT} with $\ell=k+1$) and that $w\in C(\overline{T})^3$, the $L^\infty$-boudnedness of $\lproj{k}{F}$ yields
  $\norm{L^\infty(F)}{\widehat{w}_T \cdot n_{TF}}
  \lesssim \norm{L^\infty(F)}{w\cdot n_{TF}}
  \le
  \norm{L^\infty(T)^d}{w}$
  and thus
  \[
  \begin{aligned}
    |\term_2(T)|
    &\lesssim \norm{L^{\infty}(T)^d}{w}
    \sum_{F\in \mathcal{F}_T} \norm{L^2(F)^d}{\widehat{v}_F - \widehat{v}_T}
    \norm{L^2(F)^d}{z_F -  z_T}
    \\
    &\lesssim
    h_T^{k + \frac12}
    \norm{L^{\infty}(T)^d}{w}^{\frac12} \seminorm{H^{k+1}(T)^d}{v}
    \left(
    \norm{L^{\infty}(T)^d}{w}
    \sum_{F\in \mathcal{F}_T} \norm{L^2(F)^d}{z_F -  z_T}^2
    \right)^{\frac12},
  \end{aligned}
  \]
  where the conclusion follows by applying the Cauchy--Schwarz inequality to the sum over faces,
  recalling~\eqref{eq:H1-boundedness:boundary:intermediate}, and invoking the approximation properties of $\lproj{k}{T}$ together with~\eqref{eq:RTN_bound_F} with $(p,q) = (2,k)$ to write
  \begin{equation}\label{eq:est.vF.vT}
    \norm{L^2(F)^d}{\widehat{v}_F - \widehat{v}_T}\lesssim h_T^{k + \frac12}\seminorm{H^{k+1}(T)^d}{v}.
  \end{equation}
  We next proceed differently depending on the local regime as identified by the local Péclet number.
  Recalling the definition \eqref{eq:chi.alpha} of $\chi$ to estimate the term in parentheses, we readily obtain
  \[
  \text{
    $|\term_2(T)| \lesssim \chi^{\frac12} h_T^{k + \frac12}
    \norm{L^{\infty}(T)^d}{w}^{\frac12} \seminorm{H^{k+1}(T)^d}{v} \, \seminorm{\alpha,T}{\underline{z}_T}$
    if $\Pe_T > 1$.
  }
  \]
  If, on the other hand, $\Pe_T \le 1$, we notice that
  \[
  \norm{{L^\infty(T)^d}}{w}
  \sum_{F \in \FT} \norm{L^2(F)^d}{z_F -  z_T}^2
  \overset{\eqref{eq:PeT}}\le
  \eta \Pe_T h_T^{-1} \sum_{F \in \FT} \norm{L^2(F)^d}{z_F -  z_T}^2
  \overset{\eqref{eq:norm.1.h}}\le
  \eta \Pe_T \norm{1,T}{\underline{z}_T}^2,
  \]
  leading to the estimate
  \[
  \text{
    $|\term_2(T)|
    \lesssim
    h_T^{k + \frac12} \norm{L^{\infty}(T)^d}{w}^{\frac12}
    \Pe_T^{\frac12} \seminorm{H^{k+1}(T)^d}{v} \, \eta^{\frac12} \norm{1,T}{\underline{z}_T}$
    if $\Pe_T \le 1$.
  }
  \]
  The above regime-dependent bounds together with a Cauchy--Schwarz inequality on the sum over elements give
  \begin{equation}\label{eq:th:consistency:T3}     
    |\term_2|
    \lesssim
    (1 + \chi)^{\frac12}
      \left[
        \sum_{T \in \Th}
        h_T^{2k + 1} \| w\|_{ L^{\infty}(T)^d} \min(1,\Pe_T) |  v|_{ H^{k+1}(T)^d}^2
        \right]^{\frac12}
    \left(
    | \underline{ z}_h |_{ {\alpha},h}^2
    + \eta \| \underline{ z}_h \|_{1,h}^2
    \right)^{\frac12}.
  \end{equation}
  Gathering the estimates \eqref{eq:th:consistency:T1+T2} and \eqref{eq:th:consistency:T3} yields \eqref{eq:E.conv:estimate}.
  \smallskip

  The estimate \eqref{eq:j.alpha:estimate} is obtained in the same way as the estimate for $\term_2$.
  Specifically, using Cauchy--Schwarz inequalities together with \eqref{eq:est.vF.vT}, we write
  \[
  |j_{ {\alpha},h}(\hu{v}_h, \underline{z}_h)|
  \le \sum_{T \in \Th} \alpha_T^{\frac12} h_T^{k + \frac12} |  v |_{ H^{k+1}(T)^d}
  \left(
  \alpha_T \sum_{F \in \FT} \|  z_F -  z_T \|_{ L^2(F)^d}^2
  \right)^{\frac12}
  \]
  Proceeding as above with $\alpha_T$ replacing $\| w \|_{ L^\infty(T)^d}$, we have the following regime-dependent estimate for the term in parentheses:
  \[
  \left(
  \alpha_T \sum_{F \in \FT} \|  z_F -  z_T \|_{ L^2(F)^d}^2
  \right)^{\frac12}
  \lesssim \begin{cases}
    | \underline{z}_T |_{ {\alpha},T} & \text{if $\Pe_T > 1$},
    \\
    \eta^{\frac12} \Pe_T^{\frac12} \| \underline{z}_T \|_{1,T} & \text{if $\Pe_T \le 1$}.
  \end{cases}
  \]
  The above bounds together with Cauchy--Schwarz inequalities give~\eqref{eq:j.alpha:estimate}.
\end{proof}

\begin{proof}[Proof of Corollary~\ref{cor:convergence.rates}]
  We start by recalling that $\Theta<\infty$ by the continuity and strict-positivity assumptions.
  The estimate \eqref{eq:convergence.rate} is then obtained from \eqref{eq:basic.estimate} by writing the definition~\eqref{eq:Eh} of the global consistency error with $(\vh,\wh) = (\eu,\eb)$ and estimating the terms in the right-hand side using, respectively:
  \eqref{eq:estimate.Err.time} with
  $(w,\vh) = (u,\vh)$ and %
  $(w,\vh) = (b,\wh)$;
  \eqref{eq:estimate.Err.diff.h} with
  $(w,\vh) = (u,\vh)$ %
  and $(w,\vh) = (b,\wh)$;
  \eqref{eq:E.conv:estimate} with
  $(\eta,\alpha)=(\nu,\beta)$ for the first two terms and $(\eta,\alpha)=(\mu,\gamma)$ for the last two, and
  $(w,v,\underline{z}_h) = (u,u,\vh)$, %
  $(w,v,\underline{z}_h) = (b,b,\vh)$, %
  $(w,v,\underline{z}_h) = (u,b,\wh)$ and %
  $(w,v,\underline{z}_h) = (b,u,\wh)$;
  \eqref{eq:j.alpha:estimate} with
  $(\eta,\alpha)=(\nu,\beta)$ and $(\eta,\alpha)=(\mu,\gamma)$, respectively, and
  $(v,w,\underline{z}_h) = (u,u,\vh)$ and %
  $(v,w,\underline{z}_h) = (b,b,\wh)$.
  Applying Cauchy--Schwarz inequalities and simplifying concludes the proof.
\end{proof}


\section*{Conclusion}

We have introduced and analyzed a novel method for the unsteady incompressible magnetohydynamics equations on convex domains based on hybrid approximations of all unknown fields.
The proposed discretization combines the flexibility and computational advantages of hybrid methods with a convection-semirobust error analysis, yielding estimates for the velocity and magnetic field that are independent of the inverse diffusion coefficients.
In addition, the method attains improved convergence orders in the diffusion-dominated regime, as identified by appropriate dimensionless local numbers of Reynolds or Hartmann type.
The theoretical analysis is supported by a comprehensive set of numerical experiments, which confirm the predicted convergence rates and robustness with respect to the physical parameters and illustrate the behavior of the method on a physically inspired magnetic lid-driven cavity test.
These results demonstrate that the proposed approach provides an accurate and efficient framework for the numerical simulation of incompressible magnetohydrodynamic flows and offer a promising basis for applications to more complex models and geometries.


\section*{Acknowledgements}

Funded by the European Union (ERC Synergy, NEMESIS, project number 101115663).
Views and opinions expressed are however those of the authors only and do not necessarily reflect those of the European Union or the European Research Council Executive Agency. Neither the European Union nor the granting authority can be held responsible for them.


\bibliographystyle{plain}
\bibliography{mhdhho}

\end{document}